\documentclass[a4paper, 11pt, reqno]{amsart}
\usepackage{amssymb}
\usepackage{amsthm}
\usepackage{bm}
\usepackage{latexsym}
\usepackage{amsmath}
\usepackage{eufrak}
\usepackage{mathrsfs}
\usepackage{caption}
\usepackage{amscd}
\usepackage[all,cmtip]{xy}
\usepackage[usenames]{color}
\usepackage{graphicx}
\usepackage{color}
\usepackage{amscd}
\usepackage{float}

\theoremstyle{plain}

\newtheorem{theorem}{Theorem}[section]
\newtheorem{proposition}[theorem]{Proposition}
\newtheorem{lemma}[theorem]{Lemma}
\newtheorem{corollary}[theorem]{Corollary}

\theoremstyle{definition}

\newcommand{\appsection}[1]{\let\oldthesection\thesection
\renewcommand{\thesection}{Appendix \oldthesection}
\section{#1}\let\thesection\oldthesection}

\newtheorem{definition}[theorem]{Definition}
\newtheorem{notation}[theorem]{Notation}

\theoremstyle{remark}

\newtheorem{remark}[theorem]{Remark}

\def\D{{\mathbb{D}}}

\def\Z{{\mathbb{Z}}}
\def\F{{\mathbb{F}}}
\def\Q{{\mathbb{Q}}}
\def\C{{\mathbb{C}}}
\def\P{{\mathbb{P}}}

\def\O{{\mathcal{O}}}

\def\X{{\mathcal{X}}}

\begin{document}

\title{Optimal bounds for T-singularities in stable surfaces}
\author[Julie Rana]{Julie Rana}
\email{ranaj@lawrence.edu}
\address{Department of Mathematics, Lawrence University, 711 E. Boldt Way, Appleton WI 54911, USA.}
\author[Giancarlo Urz\'ua]{Giancarlo Urz\'ua}
\email{urzua@mat.uc.cl}
\address{Facultad de Matem\'aticas, Pontificia Universidad Cat\'olica de Chile, Campus San Joaqu\'in, Avenida Vicu\~na Mackenna 4860, Santiago, Chile.}
%
%
%
\date{\today}

\begin{abstract} 
We explicitly bound T-singularities on normal projective surfaces $W$ with one singularity, and $K_W$ ample. This bound depends only on $K_W^2$, and it is optimal when $W$ is not rational. We classify and realize surfaces attaining the bound for each nonnegative Kodaira dimension of the minimal resolution of $W$. This answers effectiveness of bounds (see \cite{A94}, \cite{AM04}, \cite{L99}) for those surfaces.
\end{abstract}

\maketitle

\section{Introduction} \label{intro}

Koll\'ar and Shepherd-Barron introduced in \cite{KSB88} a natural compactification of the Gieseker moduli space of surfaces of general type with fixed $K^2$ and $\chi$ \cite{Gie77}, which is analogous to the Deligne-Mumford compactification of the moduli space of curves of genus $g \geq 2$ \cite{DM69}. This compactification is coarsely represented by a projective scheme \cite{K90} because of Alexeev's proof of boundedness \cite{A94} (see also \cite{AM04}). Thus we have a proper KSBA moduli space of stable surfaces, which includes classical canonical surfaces of general type. In particular, after fixing $K^2, \chi$ we have a finite list of singularities appearing on stable surfaces. It is a hard problem to write down that finite list explicitly (see \cite[Problem 1.24.3]{K17}).

Among the singularities that are allowed in stable surfaces, we have cyclic quotient singularities $\frac{1}{m}(1,q)$. These are defined as the germ at the origin of the quotient of $\C^2$ by the action $(x,y)\mapsto (\mu x, \mu^q y)$, where $\mu$ is a primitive $m$-th root of $1$, and $q$ is an integer with $0<q<m$ and gcd$(q,m)=1$. Among them, a very important class is formed by the ones which admit a $\Q$-Gorenstein smoothing \cite[Proposition 5.9]{LW86}, since they are precisely, the singularities showing up in a normal degeneration of canonical surfaces in the KSBA compactification \cite[Section 3]{KSB88}. These singularities are $\frac{1}{dn^2}(1,dna-1)$ with gcd$(n,a)=1$, and together with all Du Val singularities they are called T-singularities \cite[Section 3]{KSB88}. The $\Q$-Gorenstein smoothings of a T-singularity $\frac{1}{dn^2}(1,dna-1)$ occur in one $d$-dimensional component of its versal deformation space.

Let $W$ be a normal projective surface with one T-singularity $\frac{1}{dn^2}(1,dna-1)$ where $n>1$ (i.e. non Du Val), and $K_W$ ample. In particular $W$ is a stable surface. Assume that there are no-local-to-global obstructions to deform the singular point. Then this surface describes a codimension $d$ variety in the closure of the Gieseker moduli space of surfaces of general type with $K_W^2$ and $\chi(\O_W)$ fixed \cite{H11}. Thus for $d=1$ we obtain divisors. The purpose of this article is to optimally bound the T-singularity $\frac{1}{dn^2}(1,dna-1)$ in $W$ as a function of $K_W^2$, with no assumptions on existence of $\Q$-Gorenstein smoothings.

Let $$ \frac{dn^2}{dna-1} = b_1 - \frac{1}{b_2 - \frac{1}{\ddots - \frac{1}{b_r}}} =: [b_1, \ldots ,b_r]$$ be the Hirzebruch-Jung continued fraction associated to the T-singularity. We define its \textit{length} as $r$, and so it is the number of exceptional curves in its minimal resolution. This continued fraction has a very particular form \cite[Proposition 3.11]{KSB88}. The index of the T-singularity is $n$, and it satisfies $$ n \leq \textit{F}_{r-d}$$ where $\textit{F}_i$ is the $i$-th Fibonacci number defined by the recursion $\textit{F}_{-2}=1$, $\textit{F}_{-1}=1$, and $$\textit{F}_i = \textit{F}_{i-1} + \textit{F}_{i-2} $$ for $i \geq 0$. (This can be deduced from \cite[Lemma 3.4]{S89}.) That inequality is optimal, in the sense that equality is possible in infinitely many (and specific) cases; if $d=1$, these have the form $[3,\ldots,3,5,3,\ldots,3,2]$. Therefore, to bound T-singularities through the index, it is enough to bound $r-d$.

Let us consider the diagram $$ \xymatrix{  & X  \ar[ld]_{\pi} \ar[rd]^{\phi} &  \\ S &  & W}$$ where the morphism $\phi$ is the minimal resolution of $W$, and $\pi$ is a composition of blow-ups such that $S$ has no $(-1)$-curves. The best known bound in the literature is $$r \leq 400(K_W^2)^4$$ for $d=1$ and $S$ of general type, due to Y. Lee \cite[Theorem 23]{L99}. In \cite[Theorem 1.1]{R14} the first author gives the bound $r \leq 2$ when $d=1$, $K_W^2-K_S^2=1$, and $S$ is of general type. In this article we prove the following.

\begin{theorem}
Let $\kappa(S)$ be the Kodaira dimension of $S$.
\begin{itemize}
\item[1.] If $\kappa(S)=0$, then $r-d \leq 4 K_W^2$.

\item[2.] If $\kappa(S)=1$, then $r-d \leq 4 K_W^2-2$.

\item[3.] If $\kappa(S)=2$, then $$r-d \leq 4 (K_W^2-K_S^2)-4$$ when $K_W^2-K_S^2>1$, $r-d \leq 1$ otherwise.
\end{itemize}

In these three cases the bounds are optimal.
\label{neff}
\end{theorem}

\begin{remark}
Let $W$ be a normal projective surface with only T-singularities, and $K_W$ ample. Assume that $W$ is not rational and that there is a $\Q$-Gorenstein deformation $(W \subset \X) \to (0 \in \D)$ over a smooth curve germ $\D$ which is trivial for one non Du Val T-singularity of $W$, and a smoothing for all the rest. Thus the general fibre $W'$ has $K_{W'}$ ample, and it has one T-singularity $\frac{1}{dn^2}(1,dna-1)$ of length $r$. Then we can bound $r-d$ as in Theorem \ref{neff} since $\kappa(S) \leq \kappa(S')$, where $S'$ is the minimal model of the minimal resolution of $W'$. This can be proved by means of the stable MMP \cite{HTU17}, and the hierarchy of Kodaira dimensions in \cite[Lemma 2.4]{K92}. We remark that in any case the bound can be taken as $4K_W^2$, but one can be precise after performing MMP. See Corollary \ref{corKawa} for details. An instance of this is a $W$ with no local-to-global obstructions, as in the Lee-Park examples \cite{LP07} (see also \cite{SU14}).
\label{kawa}
\end{remark}

We observe that $d$ can be bounded by $\chi$ and $K^2$ via the log-Bogomolov-Miyaoka-Yau inequality (see e.g. \cite{La03}) as $$d -\frac{1}{dn^2} \leq 12 \chi(\O_W) - \frac{4}{3} K_W^2,$$ since $12 \chi(\O_W) = K_W^2 +\chi_{top}(W) + d-1$ (see e.g. \cite{HP2010}) \footnote{By a similar argument, the log-Bogomolov-Miyaoka-Yau inequality bounds the number of singularities on a surface $W$ with only T-singularities by $\frac{16}{9}(9\chi(\O_W)-K_W^2$). See also \cite[Theorem 10]{L99}.}. Also, $\chi(\O_W)$ can be bounded by $K_W^2$ via the generalized Noether's inequality in \cite[Theorem 2.10]{TZ92}. Thus, we are essentially bounding the length $r$ of the T-singularity as a linear function of $K_W^2$.

In the proof of such bounds, we will see that except for one specific situation, which involves a particular incidence between a $(-1)$-curve and the exceptional divisor of $\phi$ (a long diagram, see Definition \ref{longdiagram}), we have the improved bounds:

\[
r-d \leq
     \begin{cases}
       2K_W^2 &\quad\text{if } \kappa(S)=0 \\
       2K_W^2 - 1 &\quad\text{if } \kappa(S)=1 \\
       2(K_W^2 -K_S^2) - 1 &\quad\text{if } \kappa(S)=2 \\

     \end{cases}
\]
\vspace{0.3cm}

For the remaining case, where $K_S$ is not nef, we prove the following.

\begin{theorem}
Let $C$ be the exceptional divisor of $\phi$. If $K_S$ is not nef, then $S$ must be rational, and
\[
r-d \leq
     \begin{cases}
       2(K_W^2 -K_S^2) - K_S \cdot \pi(C) &\quad\text{if no long diagram}\\
       2(K_W^2 -K_S^2) + 1 - K_S \cdot \pi(C) &\quad\text{if long diagram of type I} \\
       4 (K_W^2 -K_S^2) - 2 K_S \cdot \pi(C) &\quad\text{if long diagram of type II} \\

     \end{cases}
\]
\end{theorem}
\vspace{0.3cm}

The intersection $K_S \cdot \pi(C)$ is negative, and so these inequalities depend indeed on that number. If we fix $K_S \cdot \pi(C)$ for the case of $S=\P^2$ (i.e. we fix the degree of the plane curve $\pi(C)$), then we can provide examples attaining the bound (see Remark \ref{exrat}). We can also give examples where $W$ is fixed (and so everything else except $\pi$) but $-K_{\P^2} \cdot \pi(C)$ tends to infinity; see Lemma \ref{unbounded} and the example after that. By Alexeev's boundedness, the minimal intersection number $-K_{\P^2} \cdot \pi(C)$ under Cremona transformations is bounded.

In relation to optimality, we give a classification in Section \ref{optimal examples} of the surfaces which achieve the bounds above for each nonnegative Kodaira dimension.

In Subsection \ref{k0} we classify the surfaces with $\kappa(S)=0$ attaining equality in Theorem \ref{neff}. They are special K3 and Enriques surfaces with a particular configuration of curves. In each of these cases we find a realizable example, and in two of them we have no local-to-global obstructions to deform. They produce via $\Q$-Gorenstein smoothings Godeaux surfaces with fundamental group $\Z/2$.

In Subsection \ref{k1} we list the five special types of elliptic surfaces with $\kappa=1$ which reach equality in Theorem \ref{neff}, and the corresponding configurations of curves. We realize one of the five cases, which gives construction of normal stable surfaces $W$ with one singularity $\frac{1}{25}(1,9)$, $p_g(W)=2$, $q(W)=0$, and $K_W^2=1$. There is a recent study of stable surfaces for those invariants in \cite{FPR17}, and this example seems to be new. The surface $W$ has obstructions, and so we do not know if it is $\Q$-Gorenstein smoothable.

In Subsection \ref{k2} we list surfaces with $\kappa(S)=2$ attaining equality in Theorem \ref{neff}. These are divided into four cases. We realize all of them. For the first case, which depends on a parameter $t \geq 5$, we obtain a $W$ with invariants $q(W)=0$, $p_g(W)=2t-7$, and $K_{W}^2=4(t-4)+1$. The corresponding surface $S$ satisfies $K_S^2/\chi_{\text{top}}(S)= \frac{t-4}{5t-14}$. We do not know if $W$ has $\Q$-Gorenstein smoothings. For the second case we obtain surfaces $W$ with $K_W^2=2$, $p_g=2$, $q(W)=0$, and T-singularity $\frac{1}{18\mu}(1,6\mu-1)$ for each $\mu=2,3,4,5$, where $d=2\mu$. For the third case we obtain a $W$ with $K_W^2=3$, $p_g(W)=2$, $q(W)=0$, T-singularity $\frac{1}{81}(1,35)$, and local-to-global obstructions. The surface $S$ is of general type with $K_S^2=1$. For the fourth case we obtain surfaces $W$ with $K_W^2=2$, $p_g=2$, $q(W)=0$, and T-singularity $\frac{1}{121}(1,43)$.




\subsection*{Acknowledgements}
We would like to thank Valery Alexeev, Igor Dolgachev, and Paul Hacking for several useful conversations, and for their interest in this work. After writing the present article, Jonny Evans and Ivan Smith informed us that they have found similar bounds for Wahl singularities and when $p_g>0$ in the symplectic setting \cite{ES17}. The second author was supported by the FONDECYT regular grant 1150068.

\section{Bounding} \label{bounding}

As in the introduction, let $W$ be a normal projective surface with one T-singularity $\frac{1}{dn^2}(1,dna-1)$ where $n>1$ (i.e. non Du Val), and $K_W$ ample. We consider the diagram $$ \xymatrix{  & X  \ar[ld]_{\pi} \ar[rd]^{\phi} &  \\ S &  & W}$$ where the morphism $\phi$ is the minimal resolution of $W$, and $\pi$ is a composition of $m$ blow-ups such that $S$ has no $(-1)$-curves.

We use the same notation as in \cite[Sect.2]{R14}. Let $E_i$ be the pull-back divisor in $X$ of the $i$-th point blown-up through $\pi$. Therefore, $E_i$ is a tree of $\P^1$'s, and it may not be reduced. Let $$C=C_1+\ldots+C_r$$ be the exceptional (reduced) divisor of $\phi$. We have $$K_S^2-m+r-d+1=K_W^2.$$

\begin{remark}
Throughout this paper, we will assume that $m>0$, since otherwise $r-d=K_W^2-K_S^2-1$, and this case holds in our main theorems.  
\end{remark}

\begin{proposition}
The divisor $\pi(C)$ is neither a tree of curves nor $\emptyset$. In particular $\kappa(S)=1,2$ implies $K_S \cdot \pi(C) \geq 1$. \label{hahah}
\end{proposition}

\begin{proof}
Notice that $\pi(C)=\emptyset$ implies existence of $(-1)$-curve intersecting $C$ at one (or zero) point. But the image of such a curve in $W$ would intersect $K_W$ negatively, because a $T$-singularity is log terminal.

If $\pi(C)$ is a tree of curves, then we must consider blow-ups over a smooth point of the tree or over a node of the tree. Over a smooth point of the tree we will get eventually a $(-1)$-curve intersecting at one (or zero) point $C$, which is not possible. Over a node, since $C$ is connected, we would have to eventually have again a $(-1)$-curve intersecting at one (or zero) point $C$.

If $\kappa(S)=1$, then $K_S \cdot \pi(C)=0$ would mean that $\pi(C)$ is on a fiber of the elliptic fibration. But then the general fiber would trivially intersect $K_W$, which is not possible. If $\kappa(S)=2$, then $K_S \cdot \pi(C)=0$ would mean that $\pi(C)$ is a ADE configuration or $\emptyset$, but none of them are possible.
\end{proof}

\begin{proposition}
The surface $S$ satisfies one of the following:

\begin{itemize}
\item[1.] It is rational.

\item[2.] It is either a K3 surface or an Enriques surface.

\item[3.] It has $\kappa(S)=1$ and $q(S)=0$.

\item[4.] It is of general type with $K_S^2 < K_W^2$.
\end{itemize}
\label{type}
\end{proposition}

\begin{proof}
This is essentially classification of surfaces. Say that $S$ is ruled. Then there is a $\P^1$-fibration $S \to D$ for some curve $D$. If some $C_i$ is a multiple section, then $D=\P^1$, and $S$ is rational. If no $C_i$ is a multiple section, then $C$ maps to one fiber. But then the general fiber $G$ has $G \cdot K_S=-2$, and so $G' \cdot K_W=-2$ for the strict transform $G'$ of $G$ in $W$. But $K_W$ is ample, a contradiction.

Say $S$ has $\kappa(S)=0$. If $S$ is bi-elliptic, then there is an elliptic fibration $S \to D$ over an elliptic curve $D$. But then the argument above leads to a contradiction. If $S$ is an abelian surface, then $\pi(C)=\emptyset$, but this is not possible by the previous proposition. So, by the classification of surfaces, the surface $S$ can be only K3 or Enriques.

Say $S$ has $\kappa(S)=1$. Then it has an elliptic fibration $S \to D$. But the $C_i$'s cannot be all in a fiber, because of ampleness of $K_W$ as above, and so $g(D)=0$. But then $q(S)=g(D)=0$, since $S$ is not a product (see \cite[V(12.2),III(18.2-3)]{BHPV04}).

Finally if $S$ is of general type, then Corollary \ref{weekineq} shows $K_W^2 > K_S^2$.
\end{proof}

\begin{lemma}
We have $\Big(\sum_{i=1}^m E_i \Big) \cdot \Big(\sum_{j=1}^r C_j \Big) = r - d + 2 - K_S \cdot \pi(C).$
\label{int}
\end{lemma}

\begin{proof}
This is a direct computation, using that $\sum_{i=1}^m E_i = K_X - \pi^*(K_S)$ and $K_X \cdot (\sum_{j=1}^r C_j) = r-d+2$; see \cite[Lemma 2.3]{R14}.
\end{proof}

\begin{lemma}
For any $i$, we have $ E_i \cdot \Big(\sum_{j=1}^r C_j \Big) \geq 1.$
\label{weekbound}
\end{lemma}

\begin{proof}
If $C_j \subset E_i$, then $C_j \cdot E_i = 0$ or $C_j \cdot E_i = -1$. The latter case can happen only for one $C_j$. On the other hand, we must have a $(-1)$-curve $F$ in $E_i$. Since $K_W$ is ample and the singularity in $W$ is log terminal, we must have $F \cdot \Big(\sum_{j=1}^r C_j \Big) \geq 2$. On the other hand, by Proposition \ref{hahah}, we know that $\pi(C)$ is not empty, and we have that $E_i$ is a tree. Therefore $E_i$ intersected with $\sum_{C_j \nsubseteq E_i} C_j$ is at least $2$. Therefore $ E_i \cdot \Big(\sum_{j=1}^r C_j \Big) \geq 1$.
\end{proof}

\begin{corollary}
We have $\Big(\sum_{i=1}^m E_i \Big) \cdot \Big(\sum_{j=1}^r C_j \Big) \geq m+1$. In particular $K_W^2 - K_S^2 \geq K_S \cdot \pi(C),$ and so we obtain $K_W^2 > K_S^2$ when $K_S$ is nef.
\label{weekineq}
\end{corollary}

\begin{proof}
This is Lemma \ref{weekbound} together with the observation that $E_m$ is a $(-1)$-curve, and so $E_m \cdot \Big(\sum_{j=1}^r C_j \Big) \geq 2$. For the rest, we use Lemma \ref{int}, $r-d+1-m=K_W^2-K_S^2$, and Proposition \ref{hahah}.
\end{proof}

The key for us will be to find a better lower bound for $$\Big(\sum_{i=1}^m E_i \Big) \cdot \Big(\sum_{j=1}^r C_j \Big).$$ For each $E_i$, we define the diagram $\Gamma_{E_i}$ as in \cite{R14}. The dual graph of the T-chain $C_1, \ldots, C_r$ is shown below in Figure \ref{C}.
\begin{figure}[htbp]
\includegraphics[width=3cm]{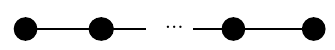}
\caption{The dual graph of $C$.}
\label{C}
\end{figure}

If $C_j\subset E_i$, we replace the $j^{\textrm{th}}$ vertex of the dual graph by a box, and denote the resulting graph by $\Gamma_{E_i}$. For instance, if  $\Gamma_{E_i}$ is as in Figure \ref{kkjkjfd}, then there are at least $4$ points of intersection among curves in the T-chain not in $E_i$ and curves in $E_i$. 

\begin{figure}[H]
\includegraphics[width=3cm]{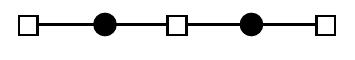}
\caption{A particular example.}
\label{kkjkjfd}
\end{figure}

If there are two or fewer points of intersection, then $\Gamma_{E_i}$ must have the form shown in Figure \ref{f3}, Figure \ref{f4}, or Figure \ref{f5}.

\begin{figure}[H]
\includegraphics[width=6.5cm]{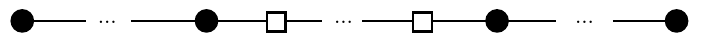}
\caption{Case (1).}
\label{f3}
\end{figure} \begin{figure}[H]
\includegraphics[width=6.5cm]{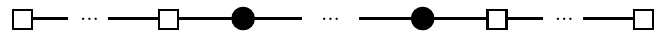}
\caption{Case (2).}
\label{f4}
\end{figure} \begin{figure}[H]
\includegraphics[width=4.5cm]{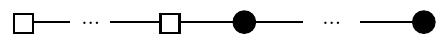}
\caption{Case (3).}
\label{f5}
\end{figure}

\begin{definition}
We say that $E_i$ has a \textbf{long diagram} if $\Gamma_{E_i}$ is as in Figure \ref{long}, and there is a $(-1)$-curve $F$ as shown in that figure (there are two types).
\label{longdiagram}
\end{definition}

\begin{figure}[htbp]
\includegraphics[width=10cm]{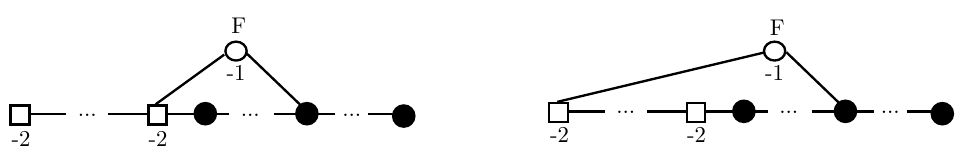}
\caption{Long diagrams of type I (left) and type II (right).}
\label{long}
\end{figure}

\begin{lemma}
An $E_i$ with $E_i \cdot \Big(\sum_{j=1}^r C_j \Big)=1$ has a long diagram.
\label{main1}
\end{lemma}

\begin{proof}
As shown above, there are three cases according to curves in $E_i$ shared by $C$.

\underline{Case (1)} is impossible because $K_W$ is ample. More precisely, this implies that a $(-1)$-curve $F$ in $E_i$ (a ``final" one) must intersect $C$ twice, and this would give either a third point of intersection with the rest of $C$ or a loop with $E_i$. But $E_i$ is a tree of $\P^1$'s.

Notice that in here we did not use the fact that $C$ is a T-chain.

\begin{figure}[htbp]
\includegraphics[width=6cm]{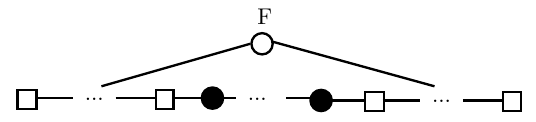}
\caption{Case (2), and $(-1)$-curve $F$.}
\label{case2}
\end{figure}

\underline{Case (2)}. In this case there is a $(-1)$-curve $F$ as in Figure \ref{case2}.

Notice that $F$ intersects one $\Box$ curve $A$ on the left and one $\Box$ curve $B$ on the right, in both cases transversally, and intersects no other curve in $E_i$. We note that either $A^2=-2$ or $B^2=-2$. Otherwise, we would need another $(-1)$-curve in $E_i$. This $(-1)$-curve would give a either a loop in $E_i$ or a third point of intersection with $C$.

Let us say $B^2=-2$. Notice that then the curve $B$ cannot have two $\Box$ neighbors, since if it did, then contracting $F$ and $B$ would give a triple point, and the $E_j$ are all simple normal crossings trees for all $j$. So we have the situation of Figure \ref{case2a}.

\begin{figure}[htbp]
\includegraphics[width=6cm]{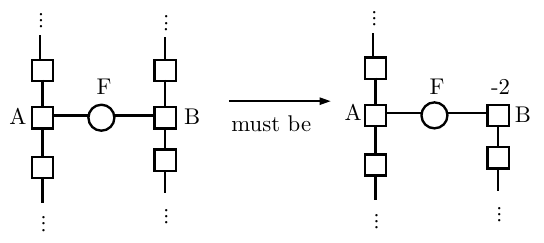}
\caption{Situation when $B^2=-2$.}
\label{case2a}
\end{figure}

Note that the curve $B$  would have multiplicity at least $2$ in $E_i$ if it had a $\bullet$ neighbor. Thus $B$ must be at end of $C$, since otherwise $E_i$ would have triple intersection with $C$. So our situation is as in Figure \ref{case2d}, for some $l\geq 0$.
We claim that in this case $A$ can have only one $\Box$ neighbor.

\begin{figure}[htbp]
\includegraphics[width=9cm]{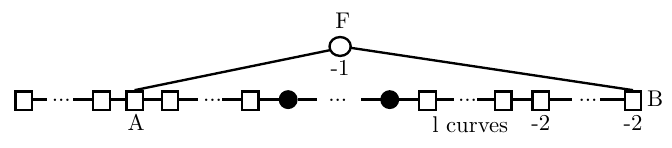}
\caption{Case 2, and $(-1)$-curve $F$.}
\label{case2d}
\end{figure}

\begin{proof}
Say that $A$ has two $\Box$ neighbors.

Suppose that after blowing down $F$, $B$, $\ldots$, $D$, as in Figure~\ref{case2d}, we have that $A$ becomes a $(-1)$-curve. If $l=0$, then $D$ has multiplicity at least $2$ in $E_i$, so this cannot happen because the intersection of $E_i$ with $C$ would be bigger than or equal to $2$. If $l>0$, then contracting the chain $F$, $B$, $\ldots$, $D$, $A$ gives a non simple normal crossing situation for $E_i$, which cannot happen.

On the other hand, suppose $A$ does not become a $(-1)$-curve after blowing down $F$, $B$, $\ldots$, $D$. Then there exists another $(-1)$-curve to continue contracting $E_i$. If this $(-1)$-curve is disjoint from the curves $F$, $B$, $\ldots$, $D$, then it is a $(-1)$-curve from the beginning in $E_i$, and so it intersects the black dots (otherwise we would generate a loop in $E_i$), a contradiction. Thus, it is not disjoint from these curves, and since $E_i$ must remain simple normal crossings at the blow downs, then $l$ must be zero. Since $l=0$, then $D$ must have multiplicity at least $2$ in $E_i$, again a contradiction.

Therefore $A$ must have only one $\Box$ neighbor, proving the claim.
\end{proof}

Notice also that $A$ cannot be at the left end of $C$, since that would give $\phi(F)\cdot K_W=0$ because $C$ is a T-chain (see Remark \ref{ends}).

\begin{remark}
Assume that the $(-1)$-curve $F$ intersects the ends of the $T$-chain $C$. Then the image of $F$ in $W$ has
$$\phi(F) \cdot K_W=-1 + 1 - \frac{dna-1 + 1}{dn^2} + 1 - \frac{dn(n-a)-1 + 1}{dn^2}=0,$$ since the discrepancies of the ends of $C$ are $-1 + \frac{dna-1 + 1}{dn^2}$ and $-1 + \frac{dn(n-a)-1 + 1}{dn^2}$. All discrepancies of $C$ can be computed as in \cite[Sect.4]{Urz13a}.
\label{ends}
\end{remark}

Therefore $A$ has a $\Box$ neighbor, and a $\bullet$ neighbor. We have two situations.

\textbf{(a)} We have $l=0$. The situation is as in Figure \ref{case2b}.

\begin{figure}[htbp]
\includegraphics[width=5cm]{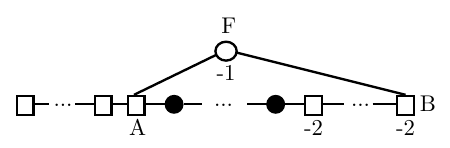}
\caption{Situation when $l=0$.}
\label{case2b}
\end{figure}

If after blowing down all $\Box$ $(-2)$-curves $B, \ldots, D$ the curve $A$ does not become a $(-1)$-curve, then we have an extra $(-2)$-curve as in Figure \ref{case2c}. This is because we need another $(-1)$-curve to continue blowing down $E_i$, and the only possibility is to come from such a situation. But then, the multiplicity in $E_i$ of the $(-2)$-curve $D$ is at least $2$, so this is not possible.

\begin{figure}[htbp]
\includegraphics[width=5cm]{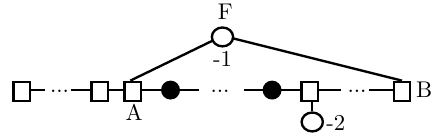}
\caption{When the $\Box$ adjacent to $A$ is not a $(-2)$-curve.}
\label{case2c}
\end{figure}

Therefore after blowing down all $(-2)$-curves $B, \ldots, D$, the curve $A$ becomes a $(-1)$-curve. If the $\Box$ adjacent to $A$ is not a $(-2)$-curve, then we need another $(-1)$-curve to continue contracting $E_i$. This means there is a $(-2)$-curve hanging as in Figure \ref{case2c}, but this is not possible as we discussed above. Therefore, the box adjacent to $A$ and all remaining $\Box$'s are at $(-2)$-curves. But $C$ is a T-configuration, and so cannot have $(-2)$-curves in both ends, a contradiction. (This is the only place in case 2 where we use that $C$ is a T-configuration.)

\textbf{(b)} We have $l>0$. If after blowing down the $(-2)$-curves $B, \ldots, D$, the curve $A$ becomes a $(-1)$-curve, then its multiplicity in $E_i$ is at least $2$. So $A$ cannot become a $(-1)$-curve. But then we need an extra $(-1)$-curve in $E_i$ to continue the contraction of $E_i$. If this $(-1)$-curve is disjoint from the curves $F$, $B, \ldots, D$, then it is a $(-1)$-curve from the beginning in $E_i$, and so it intersects the black dots (otherwise we would generate a loop in $E_i$), giving a third point of intersection of $E_i$ with $C$, a contradiction. Thus, it is not disjoint from these curves. But since $E_i$ must remain simple normal crossing at each blow-down, this forces $l=0$, again a contradiction.

Since we have proved that both situations (a) and (b) cannot occur, case (2) is impossible.

We remark that the fact that $C$ is a T-configuration was only used to eliminate the case where all $\Box$'s are $(-2)$-curves, and to eliminate the situation in which $F$ intersects both ends of $C$.
\bigskip

\underline{Case 3)}. We assume that there is $E_i$ with $$\Big(\sum_{j=1}^r C_j \Big) \cdot E_i=1.$$ In this case there must be a $(-1)$-curve $F$ that intersects a $\bullet$ curve at one point transversally, and a $\Box$ curve $A$ at one point transversally. There are no further intersections of $F$ with curves in $E_i$, because such an intersection would give a loop in $E_i$.

Notice first that $A^2=-2$. This is because if $A^2 \leq -3$, then we need another $(-1)$-curve to continue contracting $E_i$. This curve is disjoint from $F$, and so it gives from the beginning either a cycle in $E_i$ or a third point of intersection, neither of which is possible.

Also note that $A$ is adjacent to no more than one $\Box$ curve. On the contrary, suppose that $A$ is adjacent to two $\Box$ curves. Since $A^2=-2$, then $F$ has multiplicity at least $2$ in $E_i$, a contradiction with $\big(\sum_{j=1}^r C_j \big) \cdot E_i=1$.

Finally, notice that the same argument shows that all other $\Box$ curves in $C$ are also $(-2)$-curves. Otherwise, we would have either a third point of intersection of $E_i$ with $C$ or a cycle with an extra $(-1)$-curve in $E_i$. Thus we have that $E_i$ has a long diagram, as in Figure \ref{long}.
\end{proof}

\begin{lemma}
Assume that $E_i$ has a long diagram. Say that $C_1,C_2,\ldots,C_s$ are $(-2)$-curves and $C_{s+1}^2 \leq -3$. Then the number of $E_j$ with $$E_j \cdot \Big(\sum_{j=1}^r C_j \Big)=1$$ is precisely $1$ if $E_i$ is of type I, and $s$ if $E_i$ is of type II.
\label{main}
\end{lemma}

\begin{proof}
Assume $E_i$ has in its long diagram the curves $F, C_1, \ldots C_q$ where $q\leq s$. Without loss of generality, suppose that the map $\pi \colon X \to S$ starts by blowing-down $F$, and then the curves $C_1, \ldots, C_q$ from $1$ to $q$ or $q$ to $1$, depending on the type of $E_i$. Then $E_m=F$ and $E_{m-q}=F+C_1+\ldots+C_q$.

Let $E_l$ be such that $E_l \cdot C=1$ with $l < m-q$. Then $E_l$ has a long diagram by the previous lemma. So it must have as components some or all of the $(-2)$-curves $\{C_1,\ldots,C_s\}$. Here we are using that $C$ is a T-chain, so we have $(-2)$-curves only at one end. Then $E_{m-q} \subset E_l$. If the $(-1)$-curve $F'$ in the long diagram of $E_l$ is not $F$, then we have either a loop in $E_l$ or $E_l \cdot C\geq 2$. Thus $F=F'$, and so $E_l$ is of the same type as $E_i$.

Let us write $$E_l= c_1(F+C_1) + c_2 C_2 + \ldots + c_s C_s + D $$ where $c_1 \geq 1$, $c_i \geq 0$ for $i>1$, and $D$ is an effective divisor which has no $C_j$ in its support. Notice that $E_l \cdot C= c_1+ D \cdot C =1$, and so $c_1=1$ and $D \cdot C=0$. But if $D >0$, then $D$ must intersect $C$, since otherwise $D$ contains a $(-1)$-curve disjoint from $C$, a contradiction with the assumption $K_W$ ample. So, $D=0$.

If $E_l$ is of type I, then $E_i=E_l=F+C_q+\ldots+C_1$. Notice that in this case there is a unique $E_i$ such that $E_i \cdot C=1$.

If $E_l$ is of type II, then $E_l= F+C_1+C_2+\ldots+C_k$ where $1\leq k \leq s$. Therefore, we have precisely $s$ $E_j$ such that $E_j \cdot C=1$.
\end{proof}

\begin{notation}
We will use the following notation
\begin{itemize}
\item[1.] $\delta$ is the number $s$ in Lemma \ref{main} when there is a long diagram of type II, or $1$ when there is a long diagram of type I, or $0$ otherwise.

\item[2.] $\lambda:= K_S \cdot \pi(C)$.
\end{itemize}

\label{numbers}
\end{notation}

\begin{theorem}
We have $$r-d \leq 2(K_W^2- K_S^2) + \delta - \lambda.$$
\label{main2}
\end{theorem}

\begin{proof}
By Lemma \ref{int} and Lemma \ref{main} we have $$r - d +2 - \lambda =\Big(\sum_{i=1}^m E_i \Big) \cdot \Big(\sum_{j=1}^r C_j \Big) \geq 2m - \delta. $$ The result follows since $r-d+1-m=K_W^2-K_S^2$.
\end{proof}

\begin{corollary}
If there is no long diagram and $K_S$ is nef, then

\begin{itemize}
\item[1.] $\kappa(S)=0$ implies $r-d \leq 2 K_W^2$.

\item[2.] $\kappa(S)=1$ implies $r-d \leq 2 K_W^2-1$.

\item[3.] $\kappa(S)=2$ implies $r-d \leq 2 (K_W^2-K_S^2)-1$.
\end{itemize}
\label{nolong}
\end{corollary}

\begin{corollary}
If there is a long diagram of type I and $K_S$ is nef, then

\begin{itemize}
\item[1.] $\kappa(S)=0$ implies $r-d \leq 2 K_W^2 +1$.

\item[2.] $\kappa(S)=1$ implies $r-d \leq 2 K_W^2$.

\item[3.] $\kappa(S)=2$ implies $r-d \leq 2 (K_W^2-K_S^2)$.
\end{itemize}
\label{typeI}
\end{corollary}

\begin{proof}
In each case, the proof combines Theorem \ref{main2} with properties of $\lambda$ (see Proposition \ref{hahah}).
\end{proof}

We now want to estimate $s$ with respect to $r-d$ when there is a long diagram of type II.

\begin{lemma}
Assume that we have a long diagram of type II, and that $K_S$ is nef. Then

\begin{itemize}
\item[1.] $\kappa(S)=0,1$ implies $r-d \geq 2s$.

\item[2.] $\kappa(S)=2$ implies either $r-d \geq 2s+2$, or $r-d \geq 2s+1$ and $\lambda \geq 2$.
\end{itemize}

\label{nefestimate}
\end{lemma}

\begin{proof}
We divide this into three cases according to the position of the $\bullet$ curve $\Gamma$ which intersects $F$ (see Figure \ref{long} right). We denote its self-intersection by $-\alpha$. Since $C$ is a T-configuration, we have the three cases:
$$[2,\ldots,2,x_1,x_2,\ldots,x_{r-s-1},2+s],$$
$$[2,\ldots,2,3,2,\ldots,2,3+s],$$
and
$$[2,\ldots,2,4+s],$$
for some $s\ge 1$.

The two last cases are not possible for a long diagram of type II. In the last we have Remark \ref{ends} ($\phi(F)\cdot K_W=0$). For the other, we have that $\Gamma$ is the $(-3)$-curve, but $s \geq 1$ contradicts the fact that, at the end, $K_S$ is nef. So, we need to analyze only the first case. In that case, we have the following relation (see e.g. \cite[proof of Lemma 8.6]{HP2010}) $$d-3r-2=-2s - \sum_{i=1}^{r-s-1} x_i -(2+s).$$

\bigskip

\underline{$\Gamma$ at the end of $C$:} This case is impossible by Remark \ref{ends} ($\phi(F)\cdot K_W=0$).

\bigskip

\underline{$\Gamma$ intersects a $\Box$:} Notice that we have $x_1=\alpha \geq s +4$ because $K_S$ is nef, and the T-chain is of the form
$[2,\ldots,2,x_1,x_2,\ldots,x_{r-s-1},2+s]$. We reorganize the formula above as $\sum_{i=1}^{r-s-1} (x_i-2) = r-s-d+2$, and so, since $x_i-2 \geq0$, we obtain $$x_1-2 \leq r-s-d+2.$$ Because $s+4 \leq x_1$, we obtain $2s \leq r-d$.

If $S$ is of general type, then $\alpha \geq s +5$. Then we do the same and get $2s+1 \leq r-d$. If there is another $x_i$ (apart from $x_1$) with $x_i \geq 3$, then we obtain $2s+2  \leq r-d$. Let us assume that there is no such $x_i$ and that $\alpha=s+5$. Then after blowing-down $F$ and the $s$ $(-2)$-curves, we obtain a surface $S'$ such that $K_{S'} \cdot \Gamma=1$. Therefore, either $S'=S$ or there is a $(-1)$-curve intersecting only the end $(-2-s)$-curve. In either case $\lambda \geq 2$.

\bigskip

\underline{$\Gamma$ is adjacent to two $\bullet$'s:} This means $\Gamma$ does not intersect a $\Box$, and it is not at the end of $C$. Also, by adjunction and $K_S$ nef, we have $\alpha \geq s +2$. If $\alpha \geq s +3$, then $$s+3-2 + 1 \leq s+3-2 + x_1-2 \leq r-s-d+2 ,$$ which gives the desired result, $2s \leq r-d$.

The bad case to have the desired inequality is $\alpha=s+2$ and $s+1=r-s-d+2$. Then $C$ has continued fraction $[2,\ldots,2,3,2,\ldots,2,s+2,s+2]$, but then we will have a contradiction with $K_S$ nef, since some curves will become negative for canonical class. Therefore, we also have $2s \leq r-d$ in this case.

Let us consider the case $S$ of general type. Notice that $\alpha \geq s+4$ implies $r-d \geq2s+1$. So, let us assume $r-d = 2s$ and $\alpha=s+3$. Then, since $\sum_{i=1}^{r-s-1} (x_i-2) = r-s-d+2$, we have for some $\varepsilon \geq0$ $$x_1-2+(s+3-2)+\varepsilon =r-s-d+2=s+2,$$ and so $x_1=3$ and $\varepsilon=0$. Thus $x_i=2$ for all $i\neq 1$. But after contracting $F$ and $C_1, \ldots, C_s$, we obtain a cycle of $(-2)$-curves, and that is impossible in a general type surface (minimal or not). Therefore $r-d \geq 2s+1$.

Let us now consider $S$ of general type and $\alpha \geq s+5$. That implies $r-d \geq2s+2$. Let us then assume:

\textbf{(I)} $\alpha=s+3$ and $r-d=2s+1$. Then, as above, $x_1-2+s+1+\varepsilon=s+3$, and so $x_1=3$ or $4$.

If $x_1=3$, then for some $i\neq1$, we have $x_i=3$, and $x_j=2$ for all other $j$ (not corresponding to $\alpha$). We have two possibilities for $C$: $$[2,\ldots,2,3,s+3,2,\dots,2,3,s+2] \ \ \ \ \ \ [2,\ldots,2,3,2,\ldots,2,3,2,\ldots,2,s+3,s+2]$$ In the first case, we get after contracting $F, C_1, \ldots, C_s$ a cycle of two $(-2)$-curves, and that is not possible in a general type surface. In the second case, after contracting $F, C_1, \ldots, C_s$, we obtain a surface $S'$ and the configuration of curves shown in Figure \ref{f12}, where we can see self-intersections and the point $P=C_{s+1} \cap \Gamma$. If $S'= S$, then $\lambda \geq 2$ since $s \geq 1$. If $S' \neq S$, then there is a $(-1)$-curve $F'$ on $S'$. Then $F'$ intersects at most one $(-2)$-curve, and transversally, because $K_S$ is nef. So $P$ is not in $F'$. If $F'$ intersects the $(-2)$-curve chain in $S'$, then the $(-3)$-curve becomes negative for the canonical class after contraction, a contradiction. So $F'$ is disjoint from that chain. If $F'$ touches the $(-3)$-curve, then it can only be at one transversal point (since $K_S$ is nef), but then we obtain a cycle of $(-2)$-curves, which is not possible. So, $F'$ intersects the $(-2-s)$-curve, and since $K_W$ is ample (and $P$ is not in $F'$), it must be at least at two points. Then $\lambda \geq 2$.

\begin{figure}[htbp]
\includegraphics[width=3.5cm]{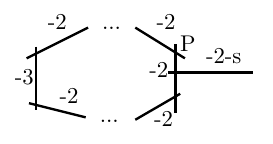}
\caption{Case \textbf{(I)} with $x_1=3$.}
\label{f12}
\end{figure}

If $x_1=4$ (so all other $x_i \neq \alpha$ are $2$), we must have the T-chain $$[2,\ldots,2,4,2,\ldots,2,s+3,2,s+2]$$ and so after contracting $F, C_1, \ldots, C_s$, we obtain a surface $S'$ and the configuration of curves shown in Figure \ref{f13}, where we can see self-intersections and the point $P=C_{s+1} \cap \Gamma$. Then the argument follows just as in the previous case, and we get $\lambda \geq 2$.

\begin{figure}[htbp]
\includegraphics[width=4cm]{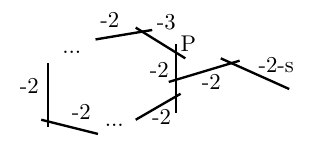}
\caption{Case \textbf{(I)} with $x_1=4$.}
\label{f13}
\end{figure}

\begin{figure}[htbp]
\includegraphics[width=3cm]{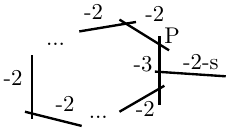}
\caption{Case \textbf{(II)}.}
\label{f14}
\end{figure}

\textbf{(II)} $\alpha=s+4$ and $r-d=2s+1$. In this case we get $x_1=3$ and $\varepsilon=0$, following same strategy. Then $C$ has the form $$[2,\ldots,2,3,2,\ldots,2,s+4,s+2]$$ and after contracting $F, C_1, \ldots, C_s$, we obtain a surface $S'$ and the configuration of curves shown in Figure \ref{f14}, where we can see self-intersections and the point $P=C_{s+1} \cap \Gamma$. Then the argument follows just as in the previous case, and we get $\lambda \geq 2$.
\end{proof}

\begin{theorem} Assume $K_S$ is nef.
\begin{itemize}
\item[1.] If $\kappa(S)=0$, then $r-d \leq 4 K_W^2$.

\item[2.] If $\kappa(S)=1$, then $r-d \leq 4 K_W^2-2$.

\item[3.] If $\kappa(S)=2$, then $$r-d \leq 4 (K_W^2-K_S^2)-4$$ when $K_W^2-K_S^2>1$, $r-d \leq 1$ otherwise.
\end{itemize}
\label{nef}
\end{theorem}

\begin{proof}
In the case $\kappa(S)=0$, we have $\lambda=0$. By Theorem \ref{main2} and Lemma \ref{nefestimate}, we have that for a long diagram of type II, $r-d \leq 4 K_W^2$. Then we compare with Corollary \ref{nolong} and Corollary \ref{typeI} to say that in any situation $r-d \leq 4 K_W^2$.

In the case $\kappa(S)=1$, we have $\lambda \geq 1$ by Proposition \ref{hahah}. By Theorem \ref{main2} and Lemma \ref{nefestimate}, we have that for a long diagram of type II, $r-d \leq 4 K_W^2-2$. Then we compare with Corollary \ref{nolong} and Corollary \ref{typeI} to say that in any situation $r-d \leq 4 K_W^2-2$.

In the case $\kappa(S)=2$, we also have $\lambda \geq 1$ by Proposition \ref{hahah}. By Theorem \ref{main2} and Lemma \ref{nefestimate}, we have that for a long diagram of type II, $r-d \leq 4 (K_W^2-K_S^2)-4$. Then we compare with Corollary \ref{nolong} and Corollary \ref{typeI} to say that in any situation $r-d \leq 4 (K_W^2-K_S^2)-4$, except in the case $K_W^2-K_S^2=1$, where we obtain $r-d \leq 2$, since in that case $4 (K_W^2-K_S^2)-4 \leq 2(K_W^2-K_S^2)$. But $K_W^2-K_S^2=1$ implies $m=r-d$, and so $m \leq 2$. Also in this case we have a long diagram of type I, and so the number of ending $(-2)$-curves in $C$ cannot exceed $1$ (otherwise $m\geq3$). So $r-d=m=2$, and $C$ has the form $[2,3,2,\ldots,4]$, but this is not possible. So for the case $K_W^2-K_S^2=1$ we must have $r-d \leq 1$.
\end{proof}

\begin{corollary}
Let $W$ be a normal projective surface with $K_W$ ample and only T-singularities. Assume that $W$ is not rational, and that there is a $\Q$-Gorenstein deformation $(W \subset \X) \to (0 \in \D)$ over a smooth curve germ $\D$ which is trivial for one non Du Val T-singularity of $W$, and a smoothing for all the rest. Thus the general fibre $W'$ has $K_{W'}$ ample, and it has one T-singularity $\frac{1}{dn^2}(1,dna-1)$ of length $r$. Then $\kappa(S) \leq \kappa(S')$, where $S'$ is the minimal model of the minimal resolution of $W'$, and so we can bound $r-d$ as in Theorem \ref{nef}.
\label{corKawa}
\end{corollary}

\begin{proof}
We resolve simultaneously the constant T-singularity in the deformation $(W \subset \X) \to (0 \in \D)$ to obtain $(W_0 \subset \X_0) \to (0 \in \D)$.
By \cite[Lemma 5.2]{HTU17} and \cite[Theorem 5.3]{HTU17}, after a possible base change, we can find a $\Q$-Gorenstein smoothing $(W_1 \subset \X_1) \to (0 \in \D)$, which is birational over $\D$ to $(W_0 \subset \X_0) \to (0 \in \D)$, such that the fibre over $0$ has only Wahl singularities \cite{W81} (i.e. non Du Val T-singularities with $d=1$), and the canonical class $K_{\X_1}$ is nef. Therefore we satisfy the conditions of \cite[Lemma 2.4]{K92}, and so we have $\kappa(\tilde{W}_1) \leq \kappa(W'_1)$ where $W'_1$ is the general fibre of $(W_1 \subset \X_1) \to (0 \in \D)$, and $\tilde{W}_1$ is the minimal resolution of $W_1$. In this way, $\kappa(S) \leq \kappa(S')$. We note that, according to \cite[Lemma 2.4]{K92}, we obtain $\kappa(S) = \kappa(S')$ if and only if $W_1$ is smooth. Finally notice that $K_{W'}$ is ample since this is a $\Q$-Gorenstein deformation with $K_W$ ample, and $\kappa(S')\geq 0$ since $\kappa(S) \geq0$ by Proposition \ref{type}. Therefore we can apply Theorem \ref{nef} to $W'$.
\end{proof}

An instance of this is a surface $W$ with no local-to-global obstructions, as in the Lee-Park examples \cite{LP07} (see also \cite{SU14}, and \cite{Urz13a} where it is done explicitly).

Based on the work done in this section and using some tricks about elliptic and rational fibrations, we have the following result when $K_S$ is not nef.

\begin{theorem}
Assume $K_S$ is not nef. Then $S$ must be rational, and
\[
r-d \leq
     \begin{cases}
       2(K_W^2 -K_S^2) - \lambda &\quad\text{if no long diagram}\\
       2(K_W^2 -K_S^2) + 1 - \lambda &\quad\text{if long diagram of type I} \\
       4 (K_W^2 -K_S^2) - 2 \lambda &\quad\text{if long diagram of type II} \\

     \end{cases}
\]
where $\lambda=K_S \cdot \pi(C)$.
\label{nonef}
\end{theorem}

\begin{proof}
By Proposition \ref{type}, we know that $S$ is rational.  We also have $$r-d \leq 2(K_W^2- K_S^2) + \delta - \lambda$$
by Theorem \ref{main2}. Therefore, it is enough to show that for the case of a long diagram of type II we have $$2\delta=2s \leq r-d.$$

Let us assume we have a long diagram of type II. We divide the analysis into three cases according to the position of the $\bullet$ curve $\Gamma$ which intersect $F$ (see Figure \ref{long} right). We denote its self-intersection by $-\alpha$. Since $C$ is a T-configuration, we have three possibilities for $C$:
$$[2,\ldots,2,x_1,x_2,\ldots,x_{r-s-1},2+s],$$
$$[2,\ldots,2,3,2,\ldots,2,3+s],$$
and
$$[2,\ldots,2,4+s],$$ for some $s\ge 1$.

The only possible one is the first case, since in the other two $\phi(F)\cdot K_W=0$. (This is another way to start the proof of Lemma \ref{nefestimate}; for the computation of discrepancies see e.g. \cite[Lemma 4.1]{Urz13a}.) The first case give us two main situations which we are going to treat separately.

\textbf{Elliptic fibration:} Assume that $C$ has continued fraction $$[2,\ldots,2,\alpha,w_1,\ldots, w_u],$$ and there is a $(-1)$-curve connecting the first $(-2)$-curve of $C$ with the curve $\Gamma$ associated to $\alpha \geq 3$. Here $u=r-s-1$ and $w_u=s+2$. Let us also assume for a contradiction that $\alpha \leq s+3$. Then, after blowing-down $F$ and all $(-2)$-curves before $\Gamma$ in $C$, we obtain a nodal curve $\Gamma'$ in a surface $S'$, which is the image of $\Gamma$, with $\Gamma'^2>0$. Let $W_1, \ldots, W_u$ be the images of the rest of the curves in $C$, so that $W_i^2=-w_i$. Let us blow-up general points in $\Gamma'$ so that the strict transform $\Gamma''$ in $S''$ has $\Gamma''^2=1$. By Riemann-Roch, the curve $\Gamma''$ defines an elliptic fibration $X' \to \P^1$, after we blow-up one base point in $S''$. The strict transform of $\Gamma''$ is a fibre. Let $W'_1$ be the strict transform of $W_1$ in $X'$. Then $W'_1$ cannot be a section. To see this, let us consider the relatively minimal fibration $X'' \to \P^1$ which has sections, and all of them are $(-1)$-curves. Therefore, there must be a $(-1)$-curve in $X'$ intersecting $W'_1$ at one point, but this would remain a $(-1)$-curve on $X$ intersecting $C$ at one point, giving a contradiction with $K_W$ ample (since this $(-1)$-curve is disjoint from the $(-1)$-sections).

Thus $W'_1, \ldots, W'_u$ are part of a fibre $G$ on $X'$, and the blow-up $X' \to S''$ is at $\Gamma'' \cap W_1$. We note that ${W'}_1^2<-2$ and ${W'}_u^2=-(s+2)<-2$. The fibre $G$ cannot be a tree because we have $(-1)$-curves in $G$, and they must touch the chain $W'_1, \ldots, W'_u$ at least twice (here we are again using that $K_W$ is ample). Therefore the only possible situation is that $G$ is a cycle, but then there is only one possible $(-1)$-curve in $G$, connecting $W'_1$ with $W'_u$, and both are $(-3)$-curves, and so $s=1$. The corresponding situation cannot be.

Therefore, in this case we have $\alpha \geq s+4$, and as in Lemma \ref{nefestimate}, we obtain $$s+4-2= s+2 \leq r-s-d+2$$ and so $2s \leq r-d$.

\textbf{Rational fibration:} Let us assume now that the T-chain has the form $[2,\ldots,2,x_1,\ldots,x_u,s+2]$ with $x_1\geq 3$, and a $(-1)$-curve $F$ connecting the first $(-2)$-curve of $C$ with an $x_j=\alpha$ with $j>1$. First, we show that $\alpha \geq s+2$.

Assume $\alpha \leq s+1$. Let us write the continued fraction of $C$ as $$[2,\ldots,2,s_1,y_1,\ldots,y_u,s_2,\alpha,s_3,z_1,\ldots,z_v]$$ where the number of $2$'s on the left is $s \geq \alpha-1$. We first show that $[y_1,\ldots,y_u]$ and $[z_1,\ldots,z_v]$ must both be empty, and so $C$ must have continued fraction $[2,\ldots,2,s_1,s_2,\alpha,s_3]$, and then we will analyze that case.

Let us say that the $2$'s on the left correspond to $C_1,\ldots,C_{\alpha-1}$. The key point of the argument is to look at $C_{\alpha-1},\ldots,C_1,F,\Gamma$. That configuration contracts to a $\P^1$ with $0$ self-intersection in a rational surface, and so it defines a genus $0$ fibration $f \colon X \to \P^1$ with $C_{\alpha-1},\ldots,C_1,F,\Gamma$ as one of its fibres. The three curves $S_i$ in $C$ which have $S_i^2=-s_i$ are sections of this fibration, since they intersect the previous fibre at one point each. The configurations of curves corresponding to $[y_1,\ldots,y_u]$ and $[z_1,\ldots,z_v]$ belong to fibres of $f$, since they are disjoint from $C_{\alpha-1},\ldots,C_1,F,\Gamma$.

We will use several times the following simple fact: In a genus $0$ fibration, a fibre which has only one $(-1)$-curve has exactly two reduced components. In particular, there cannot be $3$ sections intersecting $3$ distinct components.

Let $F'$ be a $(-1)$-curve in the fibre corresponding to $[y_1,\ldots,y_u]$. Then since $K_W$ is ample, $F'$ must intersect $C$ twice somewhere. Notice that $F'$ cannot intersect $C_1,\ldots,C_{\alpha-1},S_1,S_2$. Let us say that $F'$ intersects $S_3$, which must be transversal at one point. Then $F'$ can only intersect $[y_1,\ldots,y_u]$, since otherwise $F'$ intersects $[z_1,\ldots,z_v]$ giving that $[y_1,\ldots,y_u]$, $[z_1,\ldots,z_v]$ and $F'$ are all part of the same fibre. But $S_3$ is a section and already intersects $z_1$, so this cannot be. Let $F''$ be another $(-1)$-curve in the fibre corresponding to $[y_1,\ldots,y_u]$. Since $F'$ already intersects $S_3$, we see that $F''$ intersects $[y_1,\ldots,y_u]$ twice, a contradiction. Thus $F'$ is the only $(-1)$-curve in the fibre corresponding to $[y_1,\ldots,y_u]$, and by the fact above, this is a contradiction. Therefore $F'$ intersects $[y_1,\ldots,y_u]$ and $[z_1,\ldots,z_v]$ at one point each. Notice that there is no room for another $(-1)$-curve in that fibre. Therefore, by the fact above, we obtain a contradiction, and so there is no $[y_1,\ldots,y_u]$.

Now let $F'$ be a $(-1)$-curve in the fibre corresponding to $[z_1,\ldots,z_v]$. Then we have that either $F'$ intersects $S_1$ and $S_2$ at one point each or $F'$ intersects $S_i$ but not $S_j$, and so it also intersects $[z_1,\ldots,z_v]$. Notice that in the first case, we need to have another $(-1)$-curve in the fibre corresponding to $[z_1,\ldots,z_v]$ (by the fact above), but there is no room to have that extra $(-1)$-curve. Therefore we are in the second case, and there must exist another $(-1)$-curve $F''$ which intersects $S_j$ but not $S_i$, and intersects $[z_1,\ldots,z_v]$. There is no room for another $(-1)$-curve, and so the fibre corresponding to $[z_1,\ldots,z_v]$ must be $[1,2,\ldots,2,1]$ and so $z_i=2$ for all $i$. But $z_v=s+2$ with $s\geq 1$, a contradiction.

Thus $C$ must have continued fraction of the form $[2,\ldots,2,s_1,s_2,\alpha,s_3]$ (or $[2,\ldots,2,s_1,\alpha,s_3]$). Notice that $s_3=2+s$, $\alpha \leq s+1$, and we are assuming there are $\alpha-1$ $(-2)$-curves before $s_1$. We have a $(-1)$-curve $F$ connecting $C_1$ with $\Gamma$ which has $\Gamma^2=-\alpha$. After blowing down $F$ and $C_1, \ldots, C_{\alpha-1}$, we obtain a $\P^1$-fibration defined by the image of $\Gamma$.

Say we have $[2,\ldots,2,s_1,s_2,\alpha,s_3]$. We now can consider a model $\F_{s+2}$ by blowing down all $(-1)$-curves disjoint from the $(-s-2)$-curve $C_r$ (which comes from $C$). Each of these $(-1)$-curves should intersect transversally the $S_1$ once and the $S_2$ once, since the $S_i$ are sections. If we choose one $(-1)$-curve, then there must be another $(-1)$-curve in the same fibre which misses both $S_1$ and $S_2$. So it can only intersect $S_3$ and at one point at most, since $S_3$ is a section, a contradiction with $K_W$ ample. Thus This fibration must be already minimal, but $S_1$ and $S_2$ intersect at one point and $S_3^2=-s_3 \leq -3$, a contradiction. So this is impossible.

Then $C$ must have the form $[2,\ldots,2,s_1,\alpha,s_3]$. But after blowing-down as we just did, we have a $\P^1$-fibration where $S_1$ is a double section that, by similar reasons as above, cannot exist.

In this way, we have shown that $x_j=\alpha \geq s+2$. We also have $x_1\geq 3$. We recall that the T-chain is $[2,\ldots,2,x_1,\ldots,x_u,s+2]$. We want to show that $\sum_{i=1}^u (x_i-2) \geq s+2$, so that  $s+2 \leq r-s-d+2$ and so $2s \leq r-d$.

On the contrary, assume $x_j=\alpha=s+2$, $x_1=3$, and for all other $i\neq j$ we have $x_i=2$. Then the T-chain is $[2,\ldots,2,3,2,\ldots,2,s+2,s+2]$, and there is a $(-1)$-curve $F$ connecting $C_1$ with $\Gamma$, the $\P^1$ in $C$ with $\Gamma^2=-\alpha=-s-2$. After contracting $F$ and $C_1, \ldots, C_{\alpha-1}$, we obtain a cycle of $(-2)$-curves together with a $(-1)$-curve $\Gamma'$, the image of $\Gamma$, and a $(-s-2)$-curve $\Delta$ transversal at one point to $\Gamma'$. As before, by Riemann-Roch, that cycle (it has self-intersection $+1$) defines an elliptic fibration after blowing up one point. As before, $\Delta$ cannot be a section, but then $\Delta$ is part of a fibre. Then the only possibility that works is $\Delta$ is a $(-4)$-curve, but then $s=2$ and in this case we must have $s \geq 3$, a contradiction.

\end{proof}

\begin{remark}
In \cite[Section 5]{SU14}, we give tables describing T-singularities with $d=1$ in KSBA stable surfaces that are $\Q$-Gorenstein smoothable to simply connected surfaces of general type with $1 \leq K_W^2 \leq 4$, and $p_g=q=0$. Most of them are rational, and nearly all are T-singularities of long length. By means of the explicit MMP in \cite{HTU17}, we can realize these rational examples $W$ in such a way that $S=\P^2$; for details see \cite{Urz13}. In many cases the curve $\pi(C)$ has degree $7$. If we assume degree $7$ in Theorem \ref{nonef} (and $d=1$), we obtain that the length is at most $4K_W^2+7$ if $X$ contains a long diagram of type II, $2K_W^2+5$ if a long diagram of type I, or $2K_W^2+4$ otherwise.

The following is an example of a rational $W$ which achieves the bound with $S=\P^2$, $K_W^2=2$, $\pi(C)$ of degree $7$, and $W$ has a long diagram of type $I$. There are no local-to-global obstructions for $W$, and the singularity has continued fraction $[2,\ldots,2,12]$. Thus we have $r=9=2 K_W^2+5$.

The example comes from the table for $K^2=2$ in \cite{SU14}. The T-singularity has $d=1$, $n=10$, $a=1$. The plane curve $\pi(C)$ has degree $7$, and it contains $7$ distinct nodes, and one singularity locally of type $(y^2-x^{16})$. So, from $\P^2$ we blow-up $15$ times to resolve the singularities of the septic, and then we blow-up once more to obtain a chain of $8$ $(-2)$-curves. The strict transform of the septic has self-intersection $-12$, and we get the T-chain we want. Its contraction produces the surface $W$, where $K_W^2=9-16+9=2$. We omit the proof of ampleness and no local-to-global obstructions.
\label{exrat}
\end{remark}

We now provide an example (and a method to produce more examples) with a fixed rational $W$ but $\lambda$ arbitrarily large, by choosing $\pi$ appropriately. The key lemma is the following.

\begin{lemma}
Let $X' \to \P^1$ be a relatively minimal rational elliptic fibration with infinitely many $(-1)$-curves. Let $D$ be a section. Then there are infinitely many $(-1)$-curves $\Gamma_i$ such that $\lim_{i \to \infty} (\Gamma_i \cdot D) = \infty$. Moreover we can choose a composition of blow-downs $\sigma_i \colon X' \to \P^2$ such that the degree of $\sigma_i(D)$ approaches infinity as $i \to \infty$.
\label{unbounded}
\end{lemma}

\begin{proof}
Let us consider the divisor $B=G+D$, where $G$ is a general fibre of $X' \to \P^1$. Thus $B$ is nef and $B^2=1$, so $B$ is big and nef. Therefore there is an effective divisor $N$ and $k>>0$ such that $B-N/k$ is a $\Q$-ample divisor \cite[Lemma 2.60]{KM98}. We consider $L=B-N/k$ for that fixed $N$ and $k$. We note that the infinitely many distinct $(-1)$-curves $\Gamma_i$ are numerically independent, and so $\Gamma_i \cdot L$ is unbounded; c.f. \cite[Cor.1.19(2)]{KM98}. After rearranging the $\Gamma_i$, we may assume that $\lim_{i \to \infty} \Gamma_i \cdot L = \infty$. But a $(-1)$-curve in $X'$ is a section of $X' \to \P^1$, and so $$\Gamma_i \cdot L=1+D \cdot \Gamma_i - N \cdot \Gamma_i/k.$$ For all but finitely many $\Gamma_i$, we have $\Gamma_i \cdot N \geq 0$. Therefore, we can find an infinite sequence of $\Gamma_i$ such that $\Gamma_i \cdot D$ approaches $\infty$.

If $X' \to S=\F_l$ is a blow-down to a Hirzebruch surface, then $l=2,1,0$. This is because $X' \to \P^1$ is a relatively minimal rational elliptic fibration. Let us fix $i$ and consider as first blow-down the one with $\Gamma_i$, and then continue arbitrarily. By an elementary transformation on $\F_2$ or $\F_0$, we can assume $S=\F_1$ and the image of $D$ has a singularity of multiplicity bigger than or equal to $D \cdot \Gamma_i$, and so the same is true with the image of $D$ in the further blow-down to $\P^2$. That composition of blow-downs defines our $\sigma_i \colon X' \to \P^2$, and so the degree of $\sigma_i(D)$ is arbitrarily large.
\end{proof}

To construct an example, we again consider the list in \cite{SU14}. This example has $W$ with $K_W^2=3$, and T-singularity $\frac{1}{100^2}(1,100 \cdot 29-1)$. The continued fraction is $[4,2,6,2,6,2,2,2,4,2,2]$. Consider a rational elliptic fibration $X'$ with sections, which has $I_6$ and six $I_1$ as singular fibres. It has Mordell-Weil group of rank $3$ \cite[p.8]{Per90}. So there are infinitely many sections. We can realize its construction so that the configuration of curves in Figure \ref{f15} exists in a blow-up $X$ of $X'$ eight times. In particular we point out the special $2$-section which is a $(-4)$-curve. One can compute that there are no obstructions for $W$, and that $K_W$ is ample. Then we use Lemma \ref{unbounded} with the section $D$ in Figure \ref{f15}.

\begin{figure}[htbp]
\includegraphics[width=7cm]{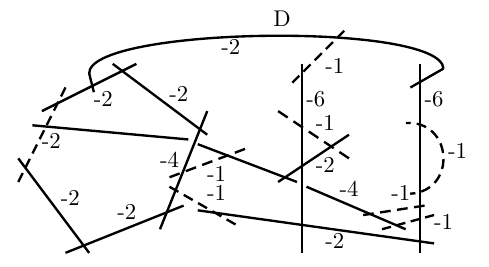}
\caption{One example which produces a situation with $\lambda \to \infty$.}
\label{f15}
\end{figure}

\section{Optimal surfaces} \label{optimal examples}

The following is a classification of the surfaces where equality is attained in Theorem \ref{nef}. In some cases, we obtain realization of these surfaces, and we analyze them in further detail.

\subsection{Case $\kappa(S)=0$} \label{k0}

\begin{theorem}
Assume that $\kappa(S)=0$ and $r-d=4K_W^2$. Then $S$ is one of the following.

\begin{itemize}
\item[(A)] A K3 surface with an elliptic fibration $f \colon S \to \P^1$ so that $\pi(C)$ is two irreducible singular fibres (with a node and a double point) and a section. All other fibres are irreducible. In this case $m=4$, $r=5$, $d=1$, $K_W^2=1$, and the T-chain is $[2,2,6,2,4]$.

\item[(B)] An Enriques surface with an elliptic fibration $f \colon S \to \P^1$ so that $\pi(C)$ is two irreducible multiple nodal fibres and a $(-2)$-curve which is a double section. In this case $m=4$, $r=5$, $d=1$, $K_W^2=1$, and the T-chain is $[2,2,6,2,4]$.

\item[(C)] An Enriques surface with an elliptic fibration $f \colon S \to \P^1$ so that $\pi(C)$ is an $I_{2k}$ double fibre and an irreducible double section with $k$ double points. The T-chain is $[2,\ldots,2,3,2,\ldots,2,2k+3,2k+2]$, and $m=3k+1$, $1 \leq K_W^2=k \leq 4$, $r=4k+1$ and $d=1$.
\end{itemize}
\label{k=0}
\end{theorem}

\begin{proof}
By Proposition \ref{type}, we know that $S$ is either a K3 surface or an Enriques surface. Also $\lambda=K_S \cdot \pi(C)=0$. By Theorem \ref{main2} and Lemma \ref{nefestimate}, we have that $r-d=4 K_W^2$ is attained when $2s=r-d$. According to the proof of Lemma \ref{nefestimate}, we must analyze two cases:

\textbf{(I)} The T-chain $C$ has continued fraction $[2,\ldots,2,s+4,2,\ldots,2,s+2]$, and there is a $(-1)$-curve $F$ intersecting the ending $(-2)$-curve and the $(-s-4)$-curve. After contracting $F$ and all the $s$ $(-2)$-curves at the end of $C$, we obtain a surface $S'$ with a self-intersection $0$ nodal rational curve together with a chain of $(s-1)$ $(-2)$-curves, and the $(-s-2)$-curve at the end. The blow-downs after that can only affect the $(-s-2)$-curve, since we cannot have a $(-1)$-curve touching the nodal or any $(-2)$-curve; otherwise $K_S$ would not be nef. In $S$ a multiple of the nodal curve is a fibre for some elliptic fibration $f \colon S \to \P^1$ c.f. \cite[VIII.17]{BHPV04}. Therefore, since $K_S\cdot \pi(C)=0$, we see that $s$ can only be $2$ or $1$. If $s=1$, then the $(-s-2)$-curve is a $(-3)$-curve. But then the image in $S$ would be $K_S$ nonzero because any $(-1)$-curve intersecting it would intersect it at least twice, since $K_W$ is ample. Thus $s=2$, and the T-chain must be $[2,2,6,2,4]$. We have either a section if $S$ is K3 or a double section if $S$ is Enriques, corresponding to the remaining $(-2)$-curve on $S$. The $(-s-2)$-curve must become a fibre, and the only possibility is to have a double point on that fibre. We have cases (A) and (B).

\textbf{(II)} The T-chain $C$ has continued fraction $[2,\ldots,2,3,2,\ldots,2,s+3,s+2]$, and there is a $(-1)$-curve $F$ intersecting the ending $(-2)$-curve and the $(-s-3)$-curve. After contracting $F$ and all the $s$ $(-2)$-curves at the end of $C$, we obtain a surface $S'$ with a cycle of $s$ $(-2)$-curves. Thus some multiple of it defines an elliptic fibration $f \colon S' \to \P^1$, and the cycle is an $I_s$ fibre. The multiplicity of $I_s$ as fibre can be either $1$ (K3) or $2$ (Enriques). Any additional $(-1)$-curve must intersect the $(-s-2)$-curve at least twice, which becomes singular, and so $I_s$ cannot have multiplicity $1$. Therefore it has multiplicity $2$, and the image of the $(-s-2)$-curve is a nodal rational curve with $k$ double points, where $s=2k$ because the intersection with canonical class is zero. We have case (C). On the other hand, in this case a quick calculation as in \cite{LP07} shows that $K_W$ is ample. Notice that $1 \leq k \leq 4$ since Enriques surfaces can have $I_l$ fibres with $1 \leq l \leq 9$ only.
\end{proof}

We can realize the three cases. First, we recall the construction of Enriques surfaces from \cite[V.23]{BHPV04}. Consider $\P^1_{x:y}\times\P_{z:w}^1$ together with the involution $i(x:y,z:w)=(x:-y,z:-w)$. Let $D_1$ and $D_2$ be intersecting fibers, both invariant under the involution $i$. Choose $p_1 \in D_1$ and $p_2 \in D_2$, neither of which is fixed by $i$, and consider a curve $B$ of bidegree $(4,4)$ which is also invariant under $i$, tangent to $D_1$ at $p_1$, and with a node at $p_2$. Notice that by choice of $B$, $D_1$, $D_2$ and the points $p_1$ and $p_2$, the curve $B$ is necessarily tangent to $D_1$ at $i(p_1)$ and has a node at $i(p_2)$.

We blowup $\P^1 \times \P^1$ at the nodes of $B$ (let $D_3$ and $D_3'$ be the exceptional curves). Let $f_1 \colon \bar S \to \P$ be the double cover of the resulting surface $\P$, branched over the proper transform of $B$. Then $\bar S$ is a $K3$ surface containing six rational $(-2)$-curves $D_1, D_1', D_2, D_2', D_3, D_3'$, the preimages of the corresponding curves on $\P$. Moreover, as described in \cite[V.23]{BHPV04}, the involution $i$ lifts to a fixed-point-free involution $j$ on $\bar S$ with $j(D_k)=D_k'$ for $k\in\{1,2,3\}$. Letting $f_2 \colon \bar S \to S$ be the corresponding unramified double cover, we obtain an Enriques surface $S$ containing three curves $D_1$, $D_2$, and $D_3$ (the images of the corresponding curves on $\P_1$). Here, $D_1$ is a nodal rational curve with $D_1^2=0$ which intersects $D_2$ in a point. The curves $D_2$ and $D_3$ intersect in two points and are $(-2)$-curves.

The space of automorphisms of $\P^1 \times\P^1$ which send the space of invariant $(4,4)$-forms to itself is $2$-dimensional. Thus, following an argument analogous to that of \cite[Lemma 3.5]{R14}, one can show that the space of such $B$ is $7$-dimensional.

We can also add the constraint on $B$ to have the intersection pattern with $D_1$ on the other $i$ invariant fibre parallel to $D_1$. This produces an Enriques surface with an elliptic fibration with two nodal multiple fibres, and a $(-2)$-curve as double section.

Finally we note that the quotient map $f_2 \colon \bar S \to S$ is defined by $2 K_S \sim 0$, and so we have {\small \begin{equation} \label{eq}
(f_2)_*\big(T_{\bar S}\big(-\log\big(\sum D_i+D_i'\big)\big)\big)=T_S\big(-\log\big(\sum D_i\big)\big)\oplus T_S\big(-\log\big(\sum D_i\big)\big)\big(-K_S \big).
\end{equation}}

We now go case by case showing existence, and computing local-to-global obstructions on $W$.


\textbf{(A)} Let us consider a K3 surface $S$ with a chain of curves formed by: nodal $0$-curve $\Gamma_1$, $(-2)$-curve $\Gamma_2$, nodal $0$-curve $\Gamma_3$. The two $0$-curves are fibres of an elliptic fibration with only irreducible fibres, and the $(-2)$-curve is a section. We can produce such an example via base change of order two from a rational elliptic fibration with only irreducible fibres, and with sections. After $m=4$ blow-ups over $S$, we obtain the Wahl chain $[2,2,6,2,4]$, and after contracting this chain we obtain $W$ with ample canonical class (we use that all fibres of $S \to \P^1$ are irreducible), and $K_W^2=1$. As in \cite{LP07}, the local-to-global obstructions lie in $H^2(S,T_S(-\log(\Gamma_1 + \Gamma_2 + \Gamma_3)))$. This cohomology space is isomorphic to $H^0(S,\Omega_S^1(\log(\Gamma_1 + \Gamma_2 + \Gamma_3)))$ by Serre duality. By the residue sequence, we have that $H^0(S,\Omega_S^1(\log(\Gamma_1 + \Gamma_2 + \Gamma_3))) \neq 0$ because $\Gamma_1$ and $\Gamma_3$ are linearly equivalent. Thus we do not know if $W$ has $\Q$-Gorenstein smoothings.

\textbf{(B)} Let us consider an Enriques surface $S$ with a chain of curves formed by: nodal $0$-curve $\Gamma_1$, $(-2)$-curve $\Gamma_2$, nodal $0$-curve $\Gamma_3$. The two $0$-curves are the two multiple fibres of an elliptic fibration with only irreducible fibres, and the $(-2)$-curve is a double section. Enriques surfaces like this exist by the construction above.

Let $f_2 \colon \bar S \to S$ be the double cover defined by $2 K_S \sim 0$. Then the preimages of $\Gamma_1$ and $\Gamma_3$ are $I_2$ fibres in an elliptic fibration on the K3 surface, and the pre-image of $\Gamma_2$ consists of two disjoint sections. By Equation (\ref{eq}), we have that $H^2(\bar S, T_{\bar S}(-\log(\sum_{i=1}^3 \Gamma_i+\Gamma_i')))$ is equal to $$H^2\Big(S, T_{S}\Big(-\log\Big(\sum_{i=1}^3 \Gamma_i\Big)\Big)\Big) \oplus H^2\Big(S, T_{S}\Big(-\log\Big(\sum_{i=1}^3 \Gamma_i\Big)\Big) \otimes \O_S\Big(-K_S\Big)\Big).$$

By Serre duality and the residue sequence, we have $$h^2\Big(\bar S, T_{\bar S}\Big(-\log\Big(\sum_{i=1}^3 \Gamma_i+\Gamma_i'\Big)\Big)\Big)=h^0\Big(\bar S,\Omega_{\bar S}^1\Big(\log\Big(\sum_{i=1}^3 \Gamma_i+\Gamma_i'\Big)\Big)\Big)=1,$$ because $\Gamma_1, \Gamma_2, \Gamma_2', \Gamma_3, \Gamma_3'$ are numerically independent but $\Gamma_1$, $\Gamma_1'$, $\Gamma_2$, $\Gamma_2'$, $\Gamma_3$, $\Gamma_3'$ are not. We also have by Serre duality and the residue sequence again that $h^2(S, T_{S}(-\log(\sum_{i=1}^3 \Gamma_i)) \otimes \O_S(-K_S))=h^0(S,\Omega_S^1(\log(\sum_{i=1}^3 \Gamma_i)))=1$, because $\Gamma_1, \Gamma_3$ are not numerically independent but $\Gamma_1, \Gamma_2$ are. Therefore $H^2(S, T_{S}(-\log(\sum_{i=1}^3 \Gamma_i)))=0$.

After $m=4$ blow-ups over $S$, we obtain the Wahl chain $[2,2,6,2,4]$, and after contracting this chain we obtain $W$ with $K_W^2=1$. Since there are no local-to-global obstructions to deform $W$, we can assume that $K_W$ is ample by smoothing possible $(-2)$-curves from the fibres of $S \to \P^1$. Thus via $\Q$-Gorenstein smoothings on $W$ we obtain Godeaux surfaces with fundamental group $\Z/2$ (using Lee-Park's method \cite{LP07}).


\textbf{(C)} Let us consider an Enriques surface $S$ with nodal $0$-curve $D_1$, a $(-2)$-curve $D_2$ intersecting $D_1$ at one point, and a $(-2)$-curve $D_3$ intersecting $D_2$ transversally at two points and disjoint from $D_1$. We constructed such Enriques surfaces above, with the same notation.

As before, let $f_2 \colon \bar S \to S$ be the double cover defined by $2 K_S \sim 0$. By Equation (\ref{eq}), we have that $H^2(\bar S, T_{\bar S}(-\log(\sum_{i=1}^3 D_i+D_i')))$ is equal to $$H^2\Big(S, T_{S}\Big(-\log\Big(\sum_{i=1}^3 D_i\Big)\Big)\Big) \oplus H^2\Big(S, T_{S}\Big(-\log\Big(\sum_{i=1}^3 D_i\Big)\Big) \otimes \O_S\Big(-K_S\Big)\Big).$$ Since the curves $D_1$, $D_1'$, $D_2$, $D_2'$, $D_3$, $D_3'$ are numerically independent in $\bar S$, the Chern map in the long exact sequence of the residue sequence is injective. Thus $H^2(\bar S, T_{\bar S}(-\log(\sum_{i=1}^3 D_i+D_i')))=0$, and so $$H^2(S, T_{S}(-\log(D_1+D_2+D_3)))=0.$$ There are no local-to-global obstructions to deform $W$.

After $m=4$ blow-ups over $S$, we obtain the Wahl chain $[2,2,3,5,4]$, and after contracting this chain we obtain $W$ with $K_W^2=1$. We again can assume $K_W$ ample because we have no obstructions to deform $W$, and so we can get rid of potential $(-2)$-curves in the fibres of $S \to \P^1$. Then a $\Q$-Gorenstein smoothing of $W$ is a Godeaux surface with fundamental group isomorphic to $\Z/2$. The surfaces $W$ describe a divisor in the moduli space of those surfaces, which matches the parameters of $B$. Here we have only considered the case $k=1$; we do not know of examples for $k>1$.

\subsection{Case $\kappa(S)=1$} \label{k1}

\begin{theorem}
Assume that $\kappa(S)=1$ and $r-d=4K_W^2-2$. Then $S$ is one of the following.
\begin{itemize}

\item[(A1)] $p_g=2$, $q=0$, and $S$ has an elliptic fibration where $\pi(C)$ is a chain consisting of an $I_1$ fibre and a $(-3)$-curve which is a section. All other fibres are irreducible. In this case $m=2$, $K_W^2=1$, and the T-chain is $[2,5,3]$.

\item[(A2)] $p_g=1$, $q=0$, and $S$ has an elliptic fibration where $\pi(C)$ is a chain consisting of an $I_1$ fibre with multiplicity $2$, and a $(-3)$-curve which is a bisection. There are no more multiple fibres, and all other fibres are irreducible. In this case $m=2$, $K_W^2=1$, and the T-chain is $[2,5,3]$.

\item[(A3)] $p_g=0$, $q=0$, and $S$ has an elliptic fibration where $\pi(C)$ is a chain consisting of an $I_1$ fibre with multiplicity $2$, and a $(-3)$-curve which is a bisection. There are two additional multiplicity $2$ fibres, and all other fibres are irreducible. In this case, $m=2$, $K_W^2=1$, and the T-chain is $[2,5,3]$.

\item[(A4)] $p_g=0$, $q=0$, and $S$ has an elliptic fibration where $\pi(C)$ is a  chain consisting of an $I_1$ fibre with multiplicity $3$, and a $(-3)$-curve which is a 3-section. There is one additional fibre with multiplicity $3$, and all other fibres are irreducible. In this case, $m=2$, $K_W^2=1$, and the T-chain is $[2,5,3]$.

\item[(A5)] $p_g=0$, $q=0$, and $S$ has an elliptic fibration where $\pi(C)$ is a  chain consisting of an $I_1$ fibre with multiplicity $4$, and a $(-3)$-curve which is a 4-section. There is one additional fibre with multiplicity $2$, and all other fibres are irreducible. In this case, $m=2$, $K_W^2=1$, and the T-chain is $[2,5,3]$.

\item[(B)] $p_g=1$, $q=0$, and $S$ has an elliptic fibration with one double fibre, where $\pi(C)$ is an $I_{2k+1}$ double fibre together with a double section, which is a rational curve with $k \geq 1$ double points. In this case $m=3k+2$, $r=4k+3$, $d=1$, $K_W^2=k+1$, and the T-chain is $[2,\ldots,2,3,2,\ldots,2,2k+4,2k+3]$.

\item[(C)] $p_g=0$, $q=0$, and $S$ has an elliptic fibration with three double fibres, where $\pi(C)$ is an $I_{2k+1}$ double fibre together with a double section which is a rational curve with $k \geq 1$ double points. In this case $m=3k+2$, $r=4k+3$, $d=1$, $K_W^2=k+1$, and the T-chain is $[2,\ldots,2,3,2,\ldots,2,2k+4,2k+3]$.

\item[(D)] $p_g=0$, $q=0$, and $S$ has an elliptic fibration with two triple fibers, where $\pi(C)$ is an $I_{s}$ triple fibre together with a triple section which is a rational curve with $k_2$ double points and $k_3$ triple points. In this case $s=2k_2+3k_3+1 \geq 2$, $m=3k_2+4k_3+2$, $r=2s+1$, $d=1$, $K_W^2=k_2+2k_3+1$, and the T-chain is $[2,\ldots,2,3,2,\ldots,2,s+3,s+2]$.

\item[(E)] $p_g=0$, $q=0$, and $S$ has an elliptic fibration with two multiple fibers of multiplicities $2$ and $4$, where $\pi(C)$ is an $I_{s}$ $4$-fibre together with a $4$-section which is a rational curve with $k_2$ double points, $k_3$ triple points, and $k_4$ $4$-tuple points. In this case $s=2k_2+3k_3+4k_4+1 \geq 2$, $m=3k_2+4k_3+5k_4+2$, $r=2s+1$, $d=1$, $K_W^2=1+k_2+2k_3+3k_4$, and the T-chain is $[2,\ldots,2,3,2,\ldots,2,s+3,s+2]$.

\end{itemize}
\label{k=1}

\end{theorem}

\begin{proof}
By Proposition \ref{type}, we know that $S$ has an elliptic fibration $S \to \P^1$. By Theorem \ref{main2} and Lemma \ref{nefestimate}, we have that $r-d=4 K_W^2-2$ is attained when $2s=r-d$ and $\lambda=K_S \cdot \pi(C) = 1$. According to the proof of Lemma \ref{nefestimate}, we must analyze two cases:

\textbf{(I)} The T-chain $C$ has continued fraction $[2,\ldots,2,s+4,2,\ldots,2,s+2]$, and there is a $(-1)$-curve $F$ intersecting the ending $(-2)$-curve and the $(-s-4)$-curve. After contracting $F$ and all the $s$ $(-2)$-curves at the end of $C$, we obtain a surface $S'$ with a self-intersection $0$ nodal rational curve together with a chain of $s-1$ $(-2)$-curves, and a $(-s-2)$-curve at the end. The blow-downs after that can only affect the $(-s-2)$-curve, since we cannot have a $(-1)$-curve touching the nodal curve or any $(-2)$-curve; otherwise $K_S$ would not be nef. We note also that the nodal $0$-curve is a fibre, possibly multiple. The $(-s-2)$-curve must become part of a fibre in $S \to \P^1$, if $s \geq 2$. 
That gives a $(-2)$-curve which is a multiple section, which is not possible because $\kappa(S)=1$. So $s=1$, and we obtain a $(-3)$-curve which is a multiple section. If it is a section, then, by the canonical formula for $K_S$, we get that $p_g(S)=2$, since $q(S)=0$. This is case (A1), all fibres are irreducible to ensure $K_W$ ample. If the $(-3)$-curve is not a section, then the nodal $0$-curve is a multiple fibre. By the canonical formula we must have $p_g(S)\leq 1$. If $p_g=1$, 
then the nodal curve has multiplicity $2$ and there are no more multiple fibres, this gives option (A2). If $p_g=0$, then the nodal curve could have multiplicity $2$ (option (A3)), $3$ (option (A4)), or $4$ (option (A5)). The fact that the remaining multiple fibres in each case are as described arises automatically by intersecting $K_S$ with the $(-3)$-curve, using the canonical bundle formula for $K_S$.

\textbf{(II)} The T-chain $C$ has continued fraction $[2,\ldots,2,3,2,\ldots,2,s+3,s+2]$, and there is a $(-1)$-curve $F$ intersecting the ending $(-2)$-curve and the $(-s-3)$-curve. Here $s \geq 2$. After contracting $F$ and all the $s$ $(-2)$-curves at the end of $C$, we obtain a surface $S'$ with a cycle of $s$ $(-2)$-curves. Thus some multiple bigger or equal to $1$ of it defines an elliptic fibration $S' \to \P^1$, and the cycle is an $I_s$ fibre. The blow-down to $S$ affects only the $(-s-2)$-curve. In $S$ the canonical class is $ K_S \equiv (p_g(S)-1) G + \sum_{i=1}^u (m_i-1)F_i$ where the $F_i$ correspond to multiple fibres, and $G$ is a general fibre. The image of the $(-s-2)$-curve in $S$ is $\pi(C_r)$, and so $K_S \cdot \pi(C_r)=\lambda=1$. We have the numerical relation $$1=K_S \cdot \pi(C_r)=(p_g(S)-1) \,G \cdot \pi(C_r) + \sum_{i=1}^u \frac{m_i-1}{m_i} \, G \cdot \pi(C_r),$$ and so we analyze the following cases:

\textbf{(IIa)} $p_g(S) \geq 2$. Then we get that $\pi(C_r)$ must be a section, and that implies $s=1$. But that is a contradiction.

\textbf{(IIb)} $p_g(S)=1$. Then $u=1$, $m_1=2$, $K_S \sim G/2$, and so $\pi(C_r)$ is a double section, and the blow-downs can only produce $k$ double points where $s=2k+1$, where $k \geq 1$. Thus $m=3k+2$, $r=4k+3$, $d=1$, $K_W^2=k$, and the T-chain is $[2,\ldots,2,3,2,\ldots,2,2k+4,2k+3]$. We are in (B).

\textbf{(IIc)} $p_g(S)=0$. Then by just using the canonical bundle formula above, we get three possible situations: the surface $S$ has an elliptic fibration with

$\centerdot$ three multiplicity $2$ fibres (one of them is $I_s$) and $\pi(C_r)$ is a double section with $k$ double points, where $s=2k+1$. In this case $m=3k+2$, $r=4k+3$, $d=1$, $K_W^2=k+1$, and the T-chain is $[2,\ldots,2,3,2,\ldots,2,2k+4,2k+3]$. This is option (C).

$\centerdot$ two multiplicity $3$ fibers (one of them is $I_s$) and $\pi(C_r)$ is a triple section with $k_1$ double points and $k_2$ triple points, where $s=1+2 k_1+3k_2$. In this case $2k_1+3k_2 \geq 1$, $m=3k_1+4k_2+2$, $r=2s+1$, $d=1$, $K_W^2=k_1+2k_2+1$, and the T-chain is $[2,\ldots,2,3,2,\ldots,2,s+3,s+2]$. This is option (D).

$\centerdot$ two multiplicity $2$ and $4$ fibers, $I_s$ is 4-fibre, and $\pi(C_r)$ is a 4-section with $k_2$ double points, $k_3$ triple points, and $k_4$ 4-tuple points,  where $s=2k_2+3k_3+4k_4+1$. In this case $2k_2+3k_3+4k_4 \geq 1$, $m=3k_2+4k_3+5k_4+2$, $r=2s+1$, $d=1$, $K_W^2=k_2+2k_3+3k_4+1$, and the T-chain is $[2,\ldots,2,3,2,\ldots,2,s+3,s+2]$. This is option (E).

\end{proof}


We give an example showing that case \textbf{(A1)} of Theorem \ref{k=1} is realizable.

Let us consider a relatively minimal rational elliptic fibration $S' \to \P^1$ with at least one nodal $I_1$ fibre, and a section. Let us take two general points in $\P^1$, and make the base change of degree $3$ branched at those points. This is equivalent to consider the $3$-cyclic cover $S \to S'$ which is branched at the two fibers corresponding to the chosen two general points in $\P^1$. Then the pull-back of a $(-1)$-curve is a $(-3)$-curve $A$, which is a section again. Notice that the pull-back of an $I_1$ is three $I_1$'s. Consider one of them, denote it by $B$. We have the induced pull-back elliptic fibration $S \to \P^1$, and $K_S \sim G$ where $G$ is a general fibre. One computes $q(S)=0$, $p_g(S)=2$, and so the Kodaira dimension of $S$ is $1$. We now blow up twice over the node of $B$, to obtain a $(-2)$-curve $C$. The configuration $A-B-C$ is $[3,5,2]$. The canonical class of $W$, the contraction of $[3,5,2]$ is ample by straightforward computation assuming that $S' \to P^1$ has only irreducible fibres. Also, $r=3$ and $K_W^2=-2+3=1$. The local-to-global obstruction of $W$ is encoded in $$H^0\big(S,\Omega_S^1(\log (B+A)) \otimes \O_{S}(K_S)\big).$$ We will show that this is not zero, and so we have obstructions and, a priori, we do not know is there is a $\Q$-Gorenstein smoothing of $W$. Notice that $$ \Omega_S^1(\log(B+A+G)) \subseteq  \Omega_S^1(\log(B+A)) \otimes  \O_{S}(K_S)$$ since $K_S \sim G$. But we can now use the residue exact sequence for $B$, $G$, and $A$ and the fact that $B$ and $G$ are linearly equivalent, to say that $h^0(S,\Omega_S^1(\log(B+A+G)))=1$.

There is a recent study of stable surfaces for these invariants in \cite{FPR17}, and this example seems to be new. We do not know if options (B), (C), (D), and (E) are realizable.





\subsection{Case $\kappa(S)=2$} \label{k2}

\begin{theorem}
Assume that $\kappa(S)=2$ and $r-d=4(K_W^2-K_S^2)-4$ if $K_W^2-K_S^2>1$, or $r-d=1$ otherwise. Then
\begin{itemize}

\item[(A)] $K_W^2-K_S^2=1$, and $\pi(C)$ is a chain formed by a rational curve $\Gamma$ with one double point and $\Gamma^2=-1$ together with a $(-2)$-curve $\Gamma_1$. We have $m=1$, and the T-chain is $[2,5]$.

\item[(B)] $K_W^2-K_S^2=1$, and $\pi(C)$ is a chain of $(-2)$-curves $\Gamma_1, \ldots, \Gamma_d$ together with a $(-3)$-curve $\Gamma$ such that $\Gamma \cdot \Gamma_i=0$ for $i \neq 2,d$, and $\Gamma \cdot \Gamma_2=\Gamma \cdot \Gamma_d=1$. We have $m=1$ and $d \geq 1$.

\item[(C)] $K_W^2-K_S^2=2$, and $\pi(C)$ is a nodal rational curve $\Gamma$ with $\Gamma^2=-1$ together with a chain of three $(-2)$-curves $\Gamma_1$, $\Gamma_2$, $\Gamma_3$ with $\Gamma \cdot \Gamma_1=1$, $\Gamma \cdot \Gamma_2=0$ and $\Gamma \cdot \Gamma_3=1$. We have $m=3$, and the T-chain is $[2,7,2,2,3]$.

\item[(D)] $K_W^2-K_S^2=2$, and $\pi(C)$ is a collection of four smooth rational curves $\Gamma_1, \Gamma_2, \Gamma_3, \Gamma_4$ where $\Gamma_i^2=-2$ for $i=1,2,4$, $\Gamma_3^2=-3$, $\Gamma_1 \cdot \Gamma_2=1$, $\Gamma_1 \cdot \Gamma_3=1$, $\Gamma_1 \cdot \Gamma_4=0$, $\Gamma_2 \cdot \Gamma_3=1$, $\Gamma_2 \cdot \Gamma_4=0$, and $\Gamma_3 \cdot \Gamma_4=2$ at two distinct points. We have $m=3$, and the T-chain is $[2,3,2,6,3]$.

\end{itemize}
\label{k=2}

\end{theorem}

\begin{proof}
Assume $K_W^2-K_S^2=1$ and $r-d=1$. We have $\lambda=K_S \cdot \pi(C) = 1$. We do not have a long diagram in this case. Since $K_W^2-K_S^2=-m+r-d+1$, we also have $m=r-d=1$. So, we have two possible T-chains: $[2,5]$ and $[2,3,2,\ldots,2,4]$. In the first case, the $(-1)$-curve must intersect the $(-5)$-curve twice, and so after contracting it we obtain what we claim in (A). In the case $[2,3,2,\ldots,2,4]$, we have that the $(-1)$-curve must intersect the $(-3)$-curve once, and the $(-4)$-curve once; otherwise there are problems with $K_S$ nef and $\kappa(S)=2$. That is case (B).

We now assume that $K_W^2-K_S^2>1$. In order to achieve an optimal bound, we must have a long diagram of type II. By Theorem \ref{main2} and Lemma \ref{nefestimate}, we have that either $\lambda=K_S \cdot \pi(C) = 1$ and $2s = r-d-2$ or $\lambda=K_S \cdot \pi(C) = 2$ and $2s = r-d-1$. But the second option gives the lower bound $4(K_W^2-K_S^2)-3$, and so it is not optimal. For the first option we have, according to the proof of Lemma \ref{nefestimate}, the following cases:

\textbf{(I)} $\alpha=s+5$, the T-chain is $[2,\ldots,2,s+5,2,\ldots,2,3,2,\ldots,2,s+2]$, and there is a $(-1)$-curve connecting the last $(-2)$-curve of $C$ with the $(-s-5)$-curve. After blowing-down that $(-1)$-curve and the $s$ $(-2)$-curves, we obtain a nodal curve with self-intersection $-1$, a $(-3)$-curve and a $(-s-2)$-curve. Since $\lambda=1$, then the $(-3)$-curve must become a $(-2)$-curve in $S$, as must the $(-s-2)$-curve. Since $s \geq 1$, this case is impossible, since the only possible scenario is to have a cycle of $(-2)$-curves, but $S$ is a surface of general type.

\textbf{(II)} $\alpha=s+6$, the T-chain is $[2,\ldots,2,s+6,2,\ldots,2,s+2]$, and there is a $(-1)$-curve connecting the last $(-2)$-curve of $C$ with the $(-s-6)$-curve. After blowing-down that $(-1)$-curve and the $s$ $(-2)$-curves, we obtain a nodal curve with self-intersection $-2$, and a $(-s-2)$-curve. The $(-s-2)$-curve cannot contribute to the intersection with $K_S$ since $\lambda=1$. So it must become a $(-2)$-curve, and so any $(-1)$-curve to be contracted must intersect it at one point. That means such a $(-1)$-curve must also intersect the nodal $(-2)$-curve, but this can only happen once, because again $\lambda=1$. Therefore $s=1$ and we have the case (C).

\textbf{(III)} $\alpha=s+5$, the T-chain is $[2,\ldots,2,3,2,\ldots,2,s+5,s+2]$, and there is a $(-1)$-curve connecting the last $(-2)$-curve of $C$ with the $(-s-5)$-curve. After blowing-down that $(-1)$-curve and the $s$ $(-2)$-curves, we obtain a curve with self-intersection $-4$, and a $(-s-2)$-curve. Since $\lambda=1$ and $S$ is of general type, then the only possible option is that the $(-4)$-curve becomes a $(-3)$-curve in $S$, and the $(-s-2)$-curve becomes a $(-2)$-curve. Then $s = 1$ and we are in case (D).

Other possible cases from the proof of Lemma \ref{nefestimate} have $\lambda >1$, so we have described all cases for which equality is attained.
\end{proof}

We now give a series of examples showing that all cases of Theorem \ref{k=2} are realizable.

\textbf{(A)} Let $t \geq 4$ be an integer. In $\P^2$, consider a line $F$, and a curve $\Gamma$ of degree $2t$ which has precisely $3$ singularities at three points of $F$: $p_1$ where it has a $(2t-6)$-simple multiple point, $p_2$ where it has a triple point locally of the form $(y^2-x^5)(y-x^2)$ where $F=(x=0)$, and $p_3$ where it has a triple point locally of the form $y(y-x^2)(y+x^2)$ where $F=(x=0)$. Such a $\Gamma$ exists, and there are several free parameters. Let $\sigma \colon Y \to \P^2$ be the blow up of $\P^2$ five times, which resolves the singularities of $\Gamma$. Let $\Gamma'$ be the strict transform. Then $\Gamma'^2=24(t-3)$, and $g(\Gamma')=2(5t-16)$. We also have $K_Y \sim -3H +E_1 +E_2 +2 E_3 +E_4 + 2E_5$ where $E_1$ is over $p_1$, $E_2$ and $E_3$ are over $p_2$, and $E_4$ and $E_5$ are over $p_3$, and $$ \sigma^*(2t H) \sim \Gamma' +(2t-6) E_1 + 3 E_2+3 E_4+6 E_3 + 6 E_5.$$ More precisely, $E_3$ and $E_5$ are $(-1)$-curves, and $E_2$ and $E_4$ are $(-2)$-curves.

Consider the double cover $f \colon X \to Y$ branched along $\Gamma' +E_2+E_4$. Then $K_X \sim f^*(K_Y + \frac{1}{2}(E_2+E_4+\Gamma'))$, and so
$$K_X \sim (t-4) f^*(L) + f^*(F') + f^*(E_1) + f^*(E_2) + f^*(E_4) $$ where $L$ is the strict transform of a general line passing through $p_1$, and $F'$ is the strict transform of $F$ under $\sigma$. We note that $f^*(E_2)$ and $f^*(E_4)$ are $(-1)$-curves. We blow them down to obtain the surface $S$. We have $K_{S} \sim (t-4) f^*(L) + f^*(F') + f^*(E_1),$ where $f^*(L)$ is a general fiber of the genus two fibration, $f^*(F')$ is a $(-2)$-curve, and $f^*(E_1)$ is a $2$-section of the fibration. The surface $S$ is minimal. The invariants of $S$ are $K_{S}^2=4(t-4)$, $\chi(\O_S)=2(t-3)$, $q(S)=h^1(Y,(t-3)\sigma^*(H)-(t-4)E_1-E_3-E_5)=0$,
and so $p_g(S)=2t-7$, $K_{S}^2=4(t-4)$. In this way, when $t>4$, we have that $S$ is of general type. When $t=4$, we have that $\kappa(S)=1$, and $K_{S}$ is a fibre of the elliptic fibration. The singular surface $W$ is obtained by blowing up the node of the nodal $(-1)$ curve, and then blowing down the chain $[5,2]$, where the $(-2)$-curve is $f^*(F')$. The nodal $(-1)$ curve is $E_3$. Notice that $E_5$ is an elliptic curve (in $S$) with self-intersection $(-1)$. For a general choice of $\Gamma'$, we have that $K_W$ is ample. We do not know if $W$ has a $\Q$-Gorenstein smoothing.

\textbf{(B)} and \textbf{(D)} Let $1 \leq \mu \leq 5$ be an integer. Consider in $\P^2$ a line $L$, four points $P_1, P_2, P_3, P_4 \in L$, and a degree $10$ plane curve $\Gamma$ having a singularity of type $(x^2-y^{2\mu})$ at $P_1$ transversal to $L$, a cusp at $P_2$ transversal to $L$, a singularity of type $(x^5-y^{10})$ at $P_3$ transversal to $L$, a simple point at $P_4$, and smooth everywhere else. For example, for $\mu=5$ we can take $L=\{ x=0 \}$, $P_1=[0,0,1]$, $P_2=[0,1,1]$, $P_3=[0,1,0]$, $P_4=[0,a,1]$, and $$\Gamma= \{ -a y^2z^8+(2a+1)y^3z^7+(-a-2)y^4z^6+y^5z^5+(a_{1, 4, 5}y^4z^5+a_{1, 4, 5}y^2z^7 $$ $$ -2a_{1, 4, 5}y^3z^6)x+((-a_{2, 3, 5}-a_{2, 2, 6})y^4z^4+a_{2, 2, 6}y^2z^6+a_{2, 3, 5}y^3z^5)x^2+(a_{3, 3, 4}y^3z^4+ $$ $$a_{3, 2, 5}y^2z^5)x^3+(a_{4, 2, 4}y^2z^4+a_{4, 3, 3}y^3z^3)x^4+(a_{5, 1, 4}yz^4+a_{5, 2, 3}y^2z^3)x^5+ $$ $$(a_{6, 1, 3}yz^3+a_{6, 2, 2}y^2z^2)x^6+a_{7, 1, 2}x^7yz^2+a_{8, 1, 1}x^8yz+a_{10, 0, 0}x^{10} =0 \}$$ for some general coefficients $a$, $a_{i,j,k}$. We resolve the $(5,5)$ singularity with two blow-ups over $P_3$, and then contract the proper transform of the tangent line at $P_3 \in \Gamma$, to obtain the Hirzebruch surface $\F_2$. The proper transforms of $L$ and $\Gamma$, which we denote by $G_0$ and $\Gamma$, are a fibre and a curve in the linear system $|5C_0 + 10 G|$ respectively, where $C_0$ is the $(-2)$-curve, and $G$ is a general fibre of $\F_2 \to \P^1$.

We note that $\Gamma^2=50$, $\Gamma \cdot K_{\F_2}=-20$, and so $p_a(\Gamma)=16$. Let $\sigma \colon Y \to \F_2$ be the composition of the two blow-ups which minimally log-resolve $G_0+\Gamma$. Let $G_1,\ldots,G_{\mu}$ be the exceptional divisors over $P_1$, and $E_1, E_3, E_2$ be the exceptional over $P_2$. Let us denote the strict transform of $\Gamma$ by $\Gamma'$. Then $\Gamma'^2=50-4\mu-4-2=44-4\mu$, and $K_Y^2=8-\mu-3=5-\mu$. Let $C'_0$ and $G'$ be the proper transforms of $C_0$ and $G$ respectively. Then $$ K_Y \sim -2C'_0 -4G' + \sum_{i=1}^{\mu} i G_i + E_1+2E_2+4E_3, $$ and $\Gamma'+E_2+C'_0$ is 2-divisible, so we have a double cover $f \colon \tilde{S} \to Y$ branched along $\Gamma'+E_2+C'_0$. By the double cover formulas, we have $$K_{\tilde{S}} \equiv f^*(C'_0 + G'+E_2+E_3),$$ $q(\tilde{S})=0$, $p_g(\tilde{S})=2$, and $K_{\tilde{S}}^2=-2$. The preimages of $E_2$ and $C'_0$ are $(-1)$-curves in $\tilde{S}$, and the preimage of $E_3$ is a $(-2)$-curve. After contracting those three curves, we obtain a surface $S$ of general type with $K_S^2=1$, $p_g(S)=2$, and $q(S)=0$. The preimage in $\tilde{S}$ of $\sum_{i=1}^{\mu} G_i$ is a chain of $2\mu-1$ $(-2)$-curves. The preimage of the strict transform of $G_0$ is a $(-4)$-curve in $\tilde{S}$. The preimage of $E_1+E_2+E_3$ becomes a chain of two $(-2)$-curves in $S$.

Therefore, if $\mu>1$, we obtain a configuration of curves as wanted for (B) with $d=2\mu+1$ $(-2)$-curves, and so we can construct $W$. We can show that $K_W$ is ample by considering the genus $2$ pencil in $S$, since all fibres except $f^{-1}(G_0+A+B)$ are irreducible by choosing general parameters for $\Gamma$. We do not know if $W$ is smoothable. When $\mu=1$, we obtain the case (D), and analogous comments hold.

\textbf{(C)} Consider again the Hirzebruch surface $\F_2$ with the $(-2)$-curve $C_0$, and the general fibre $G$. Let us fix a fibre $G_0$. As in the previous example, there are more than enough parameters to have $\Gamma \in |5C_0+10G|$ irreducible with a tacnode at some point in $G_0$, whose direction is transversal to $G_0$, tangent with multiplicity $2$ at another point of $G_0$, and smooth everywhere else. We note that $\Gamma^2=50$, $\Gamma \cdot K_{\F_2}=-20$, and so $p_a(\Gamma)=16$. Let $\sigma \colon Y \to \F_2$ be the composition of the two blow-ups which resolve $\Gamma$. Let us denote the strict transform of $\Gamma$ by $\Gamma'$. Then $K_Y \sim \sigma^*(K_{\F_2})+A+2B$ where $A$, $B$ are the exceptional curves of $\sigma$. Let $G, G_0, C_0$ be the strict transforms in $Y$ of the corresponding curves in $\F_2$.

We have that $\Gamma'+C_0$ is 2-divisible, and so there is a double cover $f \colon S' \to Y$ with $S'$ smooth. The invariants of $S'$ are $K_{S'}^2=0$, $p_g(S')=2$, and $q(S')=0$. Also $K_{S'} \sim f^*(C_0+G)$, and $f^{-1}(C_0)=C'_0$ is a $(-1)$-curve. Therefore, the blow-down $S' \to S$ of $C'_0$ is a minimal surface of general type with $K_S^2=1$. Notice that the image of $f^{-1}(G_0+A+B)$ is the wanted configuration in option (C) of Theorem \ref{k=2}. We can show that $K_W$ is ample by considering the genus $2$ fibration in $S'$, since all fibres except $f^{-1}(G_0+A+B)$ are irreducible. Also $$\Omega_{S'}^1(\log (f^{-1}(G_0+A+B)+C'_0+f^{-1}(G))) \subseteq \Omega_{S'}^1(\log (f^{-1}(G_0+A+B))) \otimes \O(K_{S'}),$$ and $\Omega_{S'}^1(\log (f^{-1}(G_0+A+B)+C'_0+f^{-1}(G)))$ has global sections by means of the residue sequence and the Chern map, because $f^{-1}(G_0+A+B) \sim f^{-1}(G)$. Thus $W$ has obstruction, and we do not know if it is smoothable.


\end{document}